\documentclass[a4paper,10pt]{amsart}
\usepackage[english]{babel}
\usepackage[utf8]{inputenc}
\usepackage[T1]{fontenc}
\usepackage{csquotes}
\usepackage[style=numeric,
	useprefix,%
	giveninits=true,%
	hyperref,%
	doi=false,%
	url=false,%
	isbn=false,%
	backend=bibtex,%
	maxbibnames=99%
	]{biblatex}
\bibliography{./BIB}

\usepackage{amssymb}
\usepackage{mathrsfs}
\usepackage{hyperref}
\usepackage[usenames,dvipsnames]{xcolor}
\hypersetup{colorlinks,%
citecolor=Black,%
filecolor=Black,%
linkcolor=Black,%
urlcolor=Black}
\usepackage{enumitem}	
\usepackage{mathtools}	
\mathtoolsset{showonlyrefs=true} 
\usepackage[normalem]{ulem}

\newcommand{\scr}[1]{\mathscr{#1}}
\newcommand{\frk}[1]{\mathfrak{#1}}
\newcommand{\N}{\mathbb{N}}	
\newcommand{\R}{\mathbb{R}}	
\newcommand{\C}{\mathbb{C}}	
\newcommand{\K}{\mathbb{K}}	
\newcommand{\Id}{\mathrm{Id}}	
\newcommand{\Diff}{\mathrm{D}}	
\newcommand{\diff}{\mathrm{d}}	
\newcommand{\did}{\,\mathrm{d}} 
\newcommand{\de}{\partial}		
\renewcommand{\div}{\operatorname{div}}	
\newcommand{\THEN}{\Rightarrow}	
\newcommand{\Lin}{\mathtt{Lin}}
\newcommand{\Ad}{\operatorname{Ad}}
\newcommand{\ad}{\operatorname{ad}}
\newcommand{\tr}{\operatorname{tr}}
\newcommand{\Ann}{\operatorname{Ann}}
\newcommand{\Hom}{\operatorname{Hom}}
\newcommand{\grad}{\nabla}
\newcommand{\laplacian}{\triangle}
\newcommand{\Lie}{\mathrm{Lie}}
\newcommand{\vol}{\mathrm{vol}}
\newcommand{\Energy}{\scr E}		
\newcommand{\Ee}{V(H)}			
\newcommand{\Qq}{\frk q}			
\newcommand{\mG}{\frk m_G}			
\newcommand{\mH}{\frk m_H}			
\newcommand{\kae}{\kappa_{\Ee}}
\newcommand{\kaq}{\kappa_{\Qq}}
\newcommand{\Tt}{\scr T}			

\theoremstyle{plain}
\newtheorem{proposition}{Proposition}[section]
\newtheorem{theorem}[proposition]{Theorem}
\newtheorem{lemma}[proposition]{Lemma}
\newtheorem{corollary}[proposition]{Corollary}
\newtheorem{thm}{Theorem}[section]

\theoremstyle{definition}
\newtheorem{definition}[proposition]{Definition}
\newtheorem{remark}[proposition]{Remark}

\theoremstyle{remark}
\newtheorem{question}{Question}

\title[Harmonic morphisms]{Harmonic morphisms\\ of sub-Riemannian Lie groups}
\author[Nicolussi Golo]{Sebastiano Nicolussi Golo}
\address[Nicolussi Golo]{Independent Researcher}
\author[Pinamonti]{Andrea Pinamonti}
\address[Andrea Pinamonti]{Department of Mathematics, University of Trento, Via Sommarive 14, 38123 Povo (Trento), Italy}
	\email[Andrea Pinamonti]{Andrea.Pinamonti@unitn.it}
\author[Warhurst]{Ben Warhurst}
\address[Ben Warhurst]{Institute of Mathematics, University of Warsaw, ul. Banacha 2, 02-097 Warsaw, Poland}
\email[Ben Warhurst]{b.warhurst@mimuw.edu.pl}
\date{\today. \IfFileExists{./.gittex}{\input{./.gittex}}{}}

\subjclass[2020]{
35B06, 
53C17, 
35H20, 
53C30, 
22F30, 
22E25.} 

\keywords{sub-Laplacian, symmetries of PDEs, sub-Riemannian Lie group, Carnot group, Heisenberg group.}

\begin{document}
\maketitle

\begin{abstract}
	We prove that harmonic morphisms between sub-Riemannian Lie groups are smooth, are symmetries of the sub-Riemannian Laplacian and  are conformal submersions that satisfy a particular PDE.
	Moreover, we give some partial results for harmonic morphisms beyond Lie groups.
	We show the existence of harmonic coordinates for harmonically homogeneous spaces, and that smooth harmonic morphisms are symmetries of the Laplacian in all sub-Riemannian manifolds.
	Finally, we compute the first variation of the horizontal energy along a harmonic morphism of sub-Riemannian Lie groups:
	it is given by a linear form on the Lie algebra of the target, the modular mismatch, which vanishes for Riemannian and Carnot targets, but surprisingly not in general.
	Therefore, unlike in the Riemannian case, harmonic morphisms of sub-Riemannian Lie groups need not be harmonic maps.
\end{abstract}

\setcounter{tocdepth}{2}
\phantomsection
\addcontentsline{toc}{section}{Contents}
\tableofcontents

\section{Introduction}

A \emph{harmonic morphism} is a continuous map that pulls back harmonic functions to harmonic functions.
Harmonic morphisms between Riemannian manifolds have been studied for a long time: we invite the reader to consult Baird and Wood's monograph \cite{zbMATH02019737} and Gudmundsson's “The Bibliography of Harmonic Morphisms”.\footnote{Available
at \url{https://www.matematik.lu.se/matematiklu/personal/sigma/harmonic/bibliography.html}
}

Harmonic morphisms between Riemannian manifolds
 are smooth maps that solve a certain PDE.
Fuglede \cite{zbMATH03530589} and Ishihara \cite{MR545705} proved that
{\it a smooth map between Riemannian manifolds is a harmonic morphism if and only if it is both a harmonic map and a conformal submersion map}.

A \emph{harmonic map} is a smooth map that is a critical point of the Dirichlet energy (see \cite[Definition 3.3.1]{zbMATH02019737}).
Harmonic functions are real-valued harmonic maps.
Harmonic maps between Riemannian manifolds are characterized by the vanishing of the \emph{tension field} (see \cite[Definition 3.2.4]{zbMATH02019737} and \cite[Theorem 3.3.3]{zbMATH02019737}).

A map $\phi:M\to N$ between Riemannian manifolds is a \emph{conformal submersion} (a.k.a., \emph{semiconformal map}) if, for every $x\in M$, the differential $\diff\phi(x):T_xM\to T_{\phi(x)}N$ is either zero, or surjective with conformal restriction $\diff\phi(x) |_{\ker(\diff\phi(x))^\perp}$ from the subspace of $T_xM$ orthogonal to the kernel of $\diff\phi(x)$ onto $T_{\phi(x)}N$
(see \cite[Definition 2.4.2]{zbMATH02019737}).
Conformal maps are conformal submersions, but also orthogonal projections and constant maps are conformal submersions.
See Section~\ref{sec670d0abe} for the definition for maps between sub-Riemannian Lie groups.
We stress that we include the degenerate case $\diff\phi(x)=0$, even though the map is not a submersion at those points.

Our goal is to extend the Fuglede--Ishihara characterization to sub-Riemannian Lie groups; see Theorem~\ref{thm6a7ae38c}.
Subelliptic and pseudoharmonic maps have been studied in several settings, including Hörmander systems, pseudohermitian and contact manifolds, and sub-Riemannian Lie groups; see, among others, \cite{zbMATH05019613, BarlettaDragomirUrakawa2001, Dong2021, zbMATH05578937, zbMATH01198583, zbMATH07887823}.
Harmonic morphisms associated with subelliptic operators were investigated in \cite{zbMATH05019613,zbMATH05578937}, and, more recently, in the pseudohermitian and Fefferman settings in \cite{DragomirEspositoLoubeau2024}.

A \emph{(measured) sub-Riemannian manifold} is determined by $(M,\tilde V,\rho,\mu)$
where $M$ is a smooth manifold, $\tilde V\subset TM$ is a bracket-generating subbundle, $\rho$ is a scalar product on $\tilde V$ and $\mu$ is a smooth measure on $M$.
We call $\tilde V$ the \emph{horizontal distribution}, and a curve $\gamma$ in $M$ is a \emph{horizontal curve} if it is absolutely continuous and $\dot\gamma(t)\in\tilde V$ for almost every $t$.
The measured sub-Riemannian structure determines a second-order hypoelliptic differential operator $\laplacian_M$, the \emph{sub-Laplacian} of $M$; see Section~\ref{subs6a7c4364}.
When $\tilde V=TM$, $\rho$ is a Riemannian metric, and $\mu$ is its Riemannian volume measure, $\laplacian_M$ is the Laplace--Beltrami operator.
Sub-Riemannian \emph{harmonic functions} are (smooth) functions $u:M\to\R$ with $\laplacian_Mu=0$.

Apart from continuity, the definition of harmonic morphism does not assume any a priori regularity of the map; our first two results concern the regularity of harmonic morphisms.

An application of~\cite{MR3912638} allows us to prove that harmonic morphisms are smooth whenever they preserve horizontal curves.

\begin{thm}\label{thm69959b8b}
	Let $M$ and $N$ be measured sub-Riemannian manifolds and let $F:M\to N$ be a continuous harmonic morphism.
	Suppose that $F$ maps horizontal curves to horizontal curves.
	Then $F$ is $C^\infty$-smooth.
\end{thm}

Our second regularity result relies on some particular assumptions concerning the target space.
A sub-Riemannian manifold $(M,\tilde V,\rho,\mu)$ is said to be \emph{harmonically homogeneous}
if there is a transitive Lie group action on $M$ by harmonic morphisms; see Section~\ref{subs6a7eff81}.
For instance, isometries that preserve the measure $\mu$ are harmonic morphisms.
So, sub-Riemannian Lie groups or, more generally, isometrically homogeneous sub-Riemannian manifolds, are harmonically homogeneous.

On harmonically homogeneous sub-Riemannian manifolds we can prove that harmonic coordinates exist, and therefore that harmonic morphisms are smooth.
We stress that the existence of harmonic coordinates on sub-Riemannian manifolds is an open problem (see \cite{MR3912638,MR3553396}) and that the definition of harmonic morphism does not assume any a-priori regularity.

Theorems~\ref{thm6a7b21cf} and~\ref{thm6a7b21d7} are proved in Section~\ref{subs6a7b219a}.

\begin{thm}\label{thm6a7b21cf}
	Let $M$ be a measured sub-Riemannian manifold that is harmonically homogeneous.
	Then every point of $M$ has harmonic coordinates, that is, a system of local coordinates whose coordinate functions are harmonic.
\end{thm}

\begin{thm}\label{thm6a7b21d7}
	Let $M$ and $N$ be measured sub-Riemannian manifolds and let $F:M\to N$ be a harmonic morphism.
	Suppose that $N$ is harmonically homogeneous.
	Then $F$ is $C^\infty$-smooth.
\end{thm}

Our second set of results describes the structure of harmonic morphisms.
In sub-Riemannian manifolds, we show that harmonic morphisms are symmetries of the sub-Riemannian Laplacian.
Focusing on sub-Riemannian Lie groups, we can give a geometric description of harmonic morphisms.

\begin{thm}\label{thm699465f3}
	Let $M$ and $N$ be measured sub-Riemannian manifolds and let $F:M\to N$ be a $C^k$-map, with $k\ge2$.
	The following statements are equivalent:
	\begin{enumerate}[label=(\roman*)]
	\item\label{thm699465f3_item1}
	$F$ is a harmonic morphism.
	\item\label{thm699465f3_item4}
	There exists $f:M\to\R$ such that, for every  $v\in C^\infty(N)$,
	\begin{equation}\label{eq699465b3}
		\laplacian_M(v\circ F) = f \cdot (\laplacian_Nv)\circ F .
	\end{equation}
	\end{enumerate}
	Moreover, the function $f$ in~\eqref{eq699465b3} is non-negative and of class $C^{k-2}$,
	and~\eqref{eq699465b3} holds for $v\in C^2(N)$.
\end{thm}

The proof of Theorem~\ref{thm699465f3} is in Section~\ref{sec699472d5}.
Our approach is based on the theory of Brelot harmonic spaces; see~\cite{MR2363343, zbMATH03224930, Fuglede2011, zbMATH05174426, zbMATH00179616}.

Our main result is a geometric characterization of harmonic morphisms between sub-Riemannian Lie groups.
A \emph{sub-Riemannian Lie group} is a Lie group endowed with a left-invariant sub-Riemannian structure.
Sub-Riemannian Lie groups are harmonically homogeneous, via the action of left-translations.

\begin{thm}\label{thm6a7ae38c}
	Let $G$ and $H$ be sub-Riemannian Lie groups, $\Omega\subset G$ open and $F:\Omega\to H$.
	The following statements are equivalent:
	\begin{enumerate}[label=(\roman*)]
	\item\label{thm6a7ae38c_item1}
	$F$ is a harmonic morphism.
	
	\item\label{thm6a7ae38c_item2}
	$F$ is $C^\infty$-smooth and there exists $\lambda:\Omega\to[0,+\infty)$
	such that, for every $v\in C^\infty(H)$,
	\begin{equation}\label{eq6a7ae3d9}
		\laplacian_G(v\circ F) = \lambda^2 \cdot (\laplacian_Hv)\circ F .
	\end{equation}

	\item\label{thm6a7ae38c_item3}
	$F$ is $C^\infty$-smooth and there exists $\lambda:\Omega\to[0,+\infty)$ such that
	$F$ is a conformal submersion of factor $\lambda$, and, for every $p\in\Omega$,

\begin{equation}\label{eq6a7c3003}
		\tr_G\bigl(\Diff^2F(p)\bigr)
		- \Diff F(p)[\mG]
		+ \lambda(p)^2\,\mH = 0 ,
	\end{equation}
	where $\mG\in V(G)$ and $\mH\in V(H)$ are the modular vectors of $G$ and $H$, respectively.
	\end{enumerate}
\end{thm}
Here $V(G)\subset\frk g=T_{1_G}G$ is the polarization at the identity, whereas $\tilde V(G)\subset TG$ is the horizontal subbundle obtained by left translation of $V(G)$.
Thus the modular vectors $\mG$ and $\mH$ in~\eqref{eq6a7c3003} belong to $V(G)$ and $V(H)$, respectively.
For the definition of conformal submersion, see Section~\ref{sec670d0abe};
the modular vector is the horizontal gradient of the logarithm of the modular function, see Section~\ref{sec670d69d9}.
The left-hand side of~\eqref{eq6a7c3003} is the \emph{harmonic morphism operator} $\Tt(F)$ of Section~\ref{subs6a7ee082}.Theorem~\ref{thm6a7ae38c} follows from Theorem~\ref{thm699465f3} and \cite[Theorem A]{2025arXiv250100576K};
we give the proof in Section~\ref{sec6a7ee0f3}.
Our last set of results, in Section~\ref{hh:sec}, compares Theorem~\ref{thm6a7ae38c} with Fuglede and Ishihara's theorem.
In the Riemannian setting, a harmonic morphism is a conformal submersion and a \emph{harmonic map}, that is, a critical point of the Dirichlet energy.
Theorem~\ref{thm6a7ae38c} says that a harmonic morphism of sub-Riemannian Lie groups is a conformal submersion that satisfies the equation $\Tt(F)=0$;
the question is whether this equation is the Euler--Lagrange equation of the horizontal energy
\begin{equation}\label{eq6a7ee1c8}
	\Energy(F) = \frac12\int_\Omega\sum_{i=1}^m\bigl|\Diff F[X_i]\bigr|_H^2\did\vol_G ,
\end{equation}
where $X_1,\dots,X_m$ is an orthonormal basis of the polarization $V(G)$.
The critical points of $\Energy$ among contact maps have been studied by Grong and Markina in~\cite[Theorem 4.2]{zbMATH07887823}.
The answer is negative, and the obstruction is a linear form on the Lie algebra of the target alone.
Namely, for every complement $\Qq$ of the polarization $V(H)$ in $\frk h$ we define in~\eqref{hh:eq:chi} the \emph{modular mismatch} $\chi\in\frk h^*$ of the splitting $\frk h=\Ee\oplus\Qq$, and we prove:
\begin{thm}\label{thm6a7ee2b1}
	Let $G$ and $H$ be sub-Riemannian Lie groups, $\Omega\subset G$ open, and $F:\Omega\to H$ a harmonic morphism, with conformal factor $\lambda$ as in Theorem~\ref{thm6a7ae38c}.
	Let $\Qq$ be a complement of $V(H)$ in $\frk h$ and $\chi\in\frk h^*$ the modular mismatch of the splitting $\frk h=\Ee\oplus\Qq$.
	Let $\mathcal F:\Omega\times(-\epsilon,\epsilon)\to H$ be a smooth contact variation of $F$: writing $F_s(p):=\mathcal F(p,s)$, assume that each $F_s$ is a contact map, that $F_0=F$, and that all the maps $F_s$ agree with $F$ outside a fixed compact subset of $\Omega$.
	If
	\[
		\varphi(p):=\left.\frac{\diff}{\diff s}\right|_{s=0}F(p)^{-1}F_s(p)
	\]
	is its variation field, then
	\begin{equation}\label{eq6a7ee2ec}
		\left.\frac{\diff}{\diff s}\right|_{s=0}\Energy(F_s)
		= \int_\Omega \lambda^2\,\chi(\varphi)\did\vol_G .
	\end{equation}
	In particular, if $\chi=0$, then $F$ is a harmonic map.
\end{thm}
The quantity $\chi$ involves neither $F$, nor $G$, nor the scalar products: it only depends on the polarized Lie algebra $(\frk h,\Ee)$ and on the choice of $\Qq$.
It vanishes for a suitable $\Qq$ when $H$ is Riemannian and when $H$ is a Carnot group, among others, and in these cases harmonic morphisms are harmonic maps, as in the Riemannian theory: see Theorem~\ref{hh:thm:balanced} and Corollary~\ref{hh:cor:examples}.
It does not vanish in general.
Indeed, let $AN$ be the simply transitive solvable Iwasawa group model of the complex hyperbolic plane $\C H^2\simeq PU(2,1)/U(2)$, and polarize $AN$ by a bracket-generating hyperplane of its Lie algebra.
In Theorem~\ref{hh:thm:counterexample} we show that the identity map of this polarized Lie group is a harmonic morphism whose horizontal energy is strictly decreased by an explicit compactly supported variation; see Section~\ref{hh:subs:counterexample}.
Thus harmonic morphisms of sub-Riemannian Lie groups need not be harmonic maps.
We close the introduction with an open question: are harmonic morphisms smooth in general?
One possible approach would be to establish the existence of harmonic coordinates on every measured sub-Riemannian manifold.

Two further questions concerning the results of Section~\ref{hh:sec} are stated in Section~\ref{hh:subs:questions}.

\medskip
\noindent\textbf{Acknowledgments.}
A.P.'s research was supported in part by the University of Trento and by the INdAM--GNAMPA 2026 Project \emph{Variational, Geometric, and Analytic Perspectives on Regularity}, CUP E53C25002010001.

\medskip
\noindent\textbf{AI disclosure.}
The authors used OpenAI's ChatGPT during the preparation of this manuscript
for language editing, literature assistance, and preliminary mathematical
discussion. Moreover, Anthropic's Claude played an important role in devising Theorem~\ref{hh:thm:counterexample}.
The authors independently verified all mathematical content and
references and assume full responsibility for the final manuscript.

\section{Brelot Harmonic spaces}
We recall the standard foundamentals of the theory of Brelot harmonic spaces.
Our main references are~\cite[Chapter 6]{MR2363343}, ~\cite{zbMATH03224930} and \cite{Fuglede2011}.
See also~\cite{zbMATH05174426}

\subsection{Brelot harmonic space}\label{subs6a7c1901}

A \emph{harmonic space} is a pair $(X,\scr H)$ where $X$ is a locally compact, Hausdorff topological space and $\scr H$ is a linear sheaf of real-valued continuous functions.

In a harmonic space $(X,\scr H)$,
an open set $U\subset X$ is \emph{$\scr H$-regular} if 
\begin{enumerate}
\item
	(Dirichlet principle)
	for every $f\in C^0(\de U)$ there exists $h\in C^0(\bar U)$ with $h|_U\in\scr H(U)$ and $h|_{\de U} = f$; and
\item
	(comparison principle)
	for every $f,g\in C^0(\de U)$ and every $h_f,h_g\in C^0(\bar U)$ with $h_f|_U\in\scr H(U)$, $h_f|_{\de U} = f$, $h_g|_U\in\scr H(U)$, and $h_g|_{\de U} = g$, if $f\le g$ then $h_f\le h_g$.
\end{enumerate}
Notice that the second condition implies the uniqueness of the function $h$ in the first condition.
The function $h$ is called the \emph{harmonic extension} of $f$ to $U$ and denoted by $H_U^f$.
The operator $C^0(\de U)\to C^0(U)$, $f\mapsto H_U^f$, is easily shown to be linear and monotone.
It follows that there is a family  $\{\mu^U_x\}_{x\in U}$ of positive measures supported on $\de U$ such that, 
for every $f\in C^0(\de U)$ and $x\in U$, we have
$H_U^f(x) = \int_{\de U} f \did\mu^U_x$.
Such measures are called \emph{harmonic measures}.

A harmonic space $(X,\scr H)$ is a \emph{Brelot harmonic space} if
\begin{enumerate}
\item
	$\scr H$-regular pre-compact domains form a basis of the topology of $X$;
\item
	(Harnack's principle)
	the limit of every increasing sequence of harmonic functions on a domain is either harmonic or identically infinite.
\end{enumerate}

%

\subsection{Hyperharmonic and superharmonic functions}

Let $(X,\scr H)$ be a Brelot harmonic space. A function
\begin{equation}
    f : X \to (-\infty,+\infty]
\end{equation}
is called \emph{hyperharmonic} if it is lower semicontinuous and, for
every relatively compact $\scr H$-regular open set $U\subset X$,
\begin{equation}
    f(x) \ge \int_{\de U} f \did\mu^U_x
    \qquad\text{for every }x\in U.
\end{equation}
A hyperharmonic function $f$ is called \emph{superharmonic} if it is
finite on a dense subset of $X$. 

\subsection{Harmonic morphism}\label{subs699474ca}
Let $(X,\scr H)$ and $(Y,\scr K)$ be two harmonic spaces.
A continuous function $F:X\to Y$ is a \emph{harmonic morphism} if, for every $V\subset Y$ open and $v\in\scr K(V)$, we have $v\circ F\in\scr H(F^{-1}V)$, where $F^{-1}V$ denotes the preimage of $V$.

\begin{proposition}[{\cite[Cor. 3.2]{zbMATH03224930}}]\label{prop6947b0c3}
	Let $M$ and $N$ be harmonic spaces, with $M$ connected, and $F:M\to N$ a non-constant harmonic morphism.
	If $V\subset N$ is open and $f:V\to\R$ is superharmonic, then $f\circ F:F^{-1}(V)\to\R$ is also superharmonic.
\end{proposition}

\section{sub-Riemannian manifolds}

\subsection{General notations}
If $M$ is a smooth manifold, we denote by $TM$ its tangent bundle
and by $\Gamma(TM)$ the space of all $C^\infty$-smooth sections of $TM$.
If $F:M\to N$ is a $C^1$ map between smooth manifolds, we denote by $\diff F:TM\to TN$ its standard differential-geometric differential.

\subsection{Measured sub-Riemannian manifold}\label{subs6a829499}
A \emph{measured sub-Riemannian manifold} is given by $(M,\tilde V,\rho,\mu)$
where $M$ is a smooth manifold, $\tilde V\subset TM$ is a bracket-generating subbundle, $\rho$ is a scalar product on $\tilde V$ and $\mu$ is a smooth measure on~$M$.
We note that there are more general definitions of sub-Riemannian manifolds, such as \cite[Definition 3.2]{MR3971262}.
In our definition, the rank of $\tilde V$, i.e., the dimension of the fiber $\tilde V|_{p}$ for $p\in M$, is constant.

\subsection{Gradient}\label{subs6a7eb8de}
Let $(M,\tilde V,\rho,\mu)$ be a measured sub-Riemannian manifold.
The \emph{(horizontal) gradient} of a $C^1$ function $f:M\to\R$ at $x\in M$ is the unique 
vector field $\tilde\grad_Mf \in \Gamma(\tilde V) \subset \Gamma(TM)$ such that 
\begin{equation}\label{eq6904c16c}
	\langle \tilde\grad_Mf(p),v \rangle_\rho = \diff f(p)[v] ,
	\quad\text{for all $v\in \tilde V_p$.}
\end{equation}

\subsection{Divergence}
Let $M$ be a smooth manifold and $\mu$ a smooth measure on $M$.
The \emph{$\mu$-divergence} of a vector field $\tilde v\in\Gamma(TM)$ is the function $\div_\mu \tilde v\in C^\infty(M)$ such that, for every $\phi\in C^\infty_c(M)$,
\begin{equation}\label{eq6904c1b2}
	\int_M (\div_\mu \tilde v) \phi \did\mu = - \int_M \diff\phi[\tilde v] \did \mu .
\end{equation}
One can express the divergence of a smooth vector field $\tilde v\in\Gamma(TM)$ 
using the Lie derivative $\mathcal L_{\tilde v}$ of the volume form $\did\mu$ as
\begin{equation}\label{eq6a8d1a70}
	(\div_{\vol_\mu}\tilde v)\did\mu
	=\mathcal L_{\tilde v}\did\mu
	=\left.\frac{\diff}{\diff t}\right|_{t=0}\Phi_t^*\did\mu ,
\end{equation}
where $\Phi_t$ is the flow of $\tilde v$, see~\cite{2025arXiv250100576K}.

We recall that, if $\tilde X\in \Gamma(TM)$ is a vector field and $f\in C^\infty(M)$, then (see~\cite[Proposition 12.32]{MR2954043} and~\cite{2025arXiv250100576K})
\begin{equation}\label{eq670d3315}
    \div_\mu(f\tilde X) = df[\tilde X] + f \div_\mu(\tilde X) .
\end{equation}

\subsection{Sub-Riemannian Laplacian}\label{subs6a7c4364}
Let $M=(M,\tilde V,\rho,\mu)$ be a measured sub-Riemannian manifold.
The \emph{sub-Riemannian Laplacian on $M$} is the differential operator 
\begin{equation}\label{eq6904c1de}
	\laplacian_M u := \div_\mu(\tilde\grad_M u),
	\qquad\forall u\in C^2(M) .
\end{equation}

\subsection{Sub-Riemannian Harmonic function}
Let $M=(M,\tilde V,\rho,\mu)$ be a measured sub-Riemannian manifold.
A function $u:M\to\R$ is \emph{harmonic} or \emph{$M$-harmonic} if $\laplacian_Mu=0$.
By the hypoellipticity of $\laplacian_M$, see~\cite{MR0222474}, harmonic functions are smooth.

\subsection{Sub-Riemannian manifolds as Brelot harmonic spaces}\label{subs6a7b1f91}

\begin{proposition}[{\cite[Theorem~8.2]{zbMATH03281121}}]\label{prop69947463}
	Let $M$ be a measured sub-Riemannian manifold and $\scr H$ the sheaf of $\laplacian_M$-harmonic functions on $M$.
	Then $(M,\scr H)$ is a Brelot harmonic space.
\end{proposition}

\begin{proposition}[{\cite[Theorem 4.2]{zbMATH03262684}}]\label{prop69947448}
	Let $M$ be a measured sub-Riemannian manifold and $u\in C^2(M)$.
	The following statements are equivalent:
	\begin{enumerate}[label=(\roman*)]
	\item
	$u$ is superharmonic (or sub-harmonic, resp.) in the abstract sense of Brelot spaces;
	\item
	$\laplacian_M u\le 0$ (or $\laplacian_Mu\ge0$, resp.).
	\end{enumerate}
\end{proposition}

Let $M$ and $N$ be measured sub-Riemannian manifolds.
By the definition in Section~\ref{subs699474ca},
a continuous function $F:M\to N$ is a harmonic morphism if for every $\Omega_N\subset N$ and every $N$-harmonic function $u:\Omega_N\to \R$, the pull back $F^*u = u\circ F:F^{-1}(\Omega_N)\to\R$ is $M$-harmonic.
Propositions~\ref{prop6947b0c3} and~\ref{prop69947448} imply that, if $F:M\to N$ is a harmonic morphism, then whenever $\laplacian_Nu\ge0$ we have $\laplacian_M(F^*u)\ge0$.

\subsection{Harmonically homogeneous sub-Riemannian manifolds}\label{subs6a7eff81}
A measured sub-Riemannian manifold $M$, is \emph{harmonically homogeneous}, if there is a Lie group $K$ with a transitive action $K\times M\to M$, such that for every $k\in K$ the map $x\mapsto kx$ is a harmonic morphism of $M$.

Isometries of $M$ that preserve the measure of $M$ commute with $\laplacian_M$, see~\cite{2025arXiv250100576K}, and thus they are harmonic morphisms.
Therefore, isometrically homogeneous measured sub-Riemannian manifolds, and in particular sub-Riemannian Lie groups with their left translations, are harmonically homogeneous.

\section{sub-Riemannian Lie groups}

For sub-Riemannian Lie groups, we use the same notation as in~\cite{2025arXiv250100576K}.
Here we will only list the main objects.
For further details on the theory of sub-Riemannian Lie groups, see~\cite{zbMATH08074076}.

\subsection{Lie groups}
Given a Lie group $G$, we identify, as vector spaces, its Lie algebra $\frk g=\Lie(G)$ with the tangent space $T_{1_G}G$ of $G$ at the identity element $1_G$ of $G$.
We denote by $L_p$ the left translation $G\to G$ by $p$: $L_p(x) = px$.
If $v:E\to\frk g$ for some $E\subset G$, then we define  $\tilde v:E\to TG$ by the following formula 
\begin{equation}
\tilde v(p) := dL_p[v] \in T_pG,
\qquad \forall p\in E .
\end{equation}
If $v\in\frk g$, then $\tilde v$ is the left-invariant vector field with value $v$ at $1$.

If $f:\Omega\to W$ is a $C^1$-function from an open set $\Omega\subset G$ to a finite-dimensional vector space $W$,
and $v:\Omega\to\frk g$ is continuous, 
 then we define $\tilde vf:\Omega\to W$ by 
\begin{equation}
\tilde vf(p) := df(p)[\tilde v(p)] = \left.\frac{\diff}{\diff t}\right|_{t=0} f(p\exp(tv(p))) .
\end{equation}

\subsection{Lie differential}
Let $G$ and $H$ be Lie groups and $\Omega\subset G$ open.
The \emph{Lie differential of order one} of a smooth function $F:\Omega\to H$ at $p\in \Omega$ is the  map $\Diff F : \Omega \to \frk h \otimes \frk{g}^* $ defined by 
\begin{equation}\label{DFdef1}
	\Diff F(p)[v] := \left.\frac{\diff}{\diff t}\right|_{t=0} F(p)^{-1}F(p\exp(tv))  ,
	\qquad\forall v\in\frk g
\end{equation}

Let $G$ and $H$ be two polarized Lie groups, $\Omega\subset G$ open
and $F:\Omega\to H$ a $C^2$ function.
Notice that the function $\Diff F:\Omega\to\frk h\otimes\frk g^*$ is a $C^1$ function valued in the Abelian Lie group $\frk h\otimes\frk g^*$.
As such, we define the \emph{Lie differential of order two} of $F$ as the Lie differential $\Diff (\Diff F)$ of $\Diff F$,
that is, as the function $\Diff^2F: \Omega \to (\frk{h} \otimes \frk g^*) \otimes \frk g^*$,
\begin{equation}\label{eq67684517}
\begin{aligned}
	\Diff^2F(p)[v,w] 
	&:= (\Diff(\Diff F)(p)[w])[v] 
	\overset{\eqref{DFdef1}}= \left.\frac{\diff}{\diff t}\right|_{t=0} ( \Diff F(p \exp(t w))[v] - \Diff F(p)[v] )\\
	&= \left.\frac{\diff}{\diff t}\right|_{t=0} \Diff F(p \exp(t w))[v] \\ 
	&= \left.\frac{\diff}{\diff t}\right|_{t=0} \left.\frac{\diff}{\diff s}\right|_{s=0} F(p\exp(tw))^{-1}F(p\exp(tw)\exp(sv)) . 
\end{aligned}
\end{equation}

Notice that $D^2F(p)[v,w]$ need not be symmetric in $[v,w]$.

\subsection{Sub-Riemannian Lie groups}\label{sssec:num2.2}
A \emph{polarized Lie group} is a pair $(G,V(G))$ where $G$ is a connected Lie group and $V(G)\subset\frk g$ is a bracket-generating subspace of the Lie algebra $\frk g$ of $G$. We call $V(G)$ the \emph{polarization} of $G$.
By \emph{bracket-generating subspace} we mean a linear subspace that Lie generates the Lie algebra.


Consequently, the bundle of left translates
\begin{equation}
\tilde V(G) := \bigcup_{p \in G} d L_p(V(G))
\end{equation}
forms a \emph{horizontal subbundle} of $TG$, which is bracket generating.

A \emph{measured sub-Riemannian Lie group} is the measured sub-Riemannian manifold determined by $(G, \tilde V(G),\langle \cdot,\cdot \rangle_G,\vol_G)$ where $(G,V(G))$ is a polarized Lie group, $\langle \cdot,\cdot \rangle_G$ is a left-invariant scalar product on $\tilde V(G)$ and $\vol_G$ is a left Haar measure on $G$.
Notice that, by left-invariance, the scalar product is determined by its value on $V(G)\subset T_{1_G}G$.

If $\Omega\subset G$ is open and $f\in C^1(\Omega)$, then the horizontal gradient $\tilde\grad_G f\in \Gamma(\tilde V(G))$ from Section~\ref{subs6a7eb8de} defines a map $\grad_Gf:\Omega\to V(G)$ such that
\begin{equation}\label{eq6a7f0b3a}
	\tilde\grad_G f(p) = \diff L_p|_{1_G} [\grad_Gf(p)] ,
\end{equation}
that is, $\grad_Gf(p)$ is the unique vector of $V(G)$ such that
$\langle \grad_Gf(p),v\rangle_G = \Diff f(p)[v]$ for every $v\in V(G)$.

\subsection{Haar measure and the modular function}\label{sec670d69d9}
Let $G$ be a sub-Riemannian Lie group
with polarization $V(G)\subset\frk g$
and left-invariant Haar measure $\vol_G$.
%
The volume form~$\did\vol_G$ is left invariant, in the sense that
\begin{equation}\label{eq6762c2f8}
	\forall a,x\in G
	\qquad
	dL_a|_x^* \did\vol_G(ax) = \did\vol_G(x) .
\end{equation}

The \emph{modular function} of a Lie group $G$ is the function $\mu_G:G\to (0,\infty)$ given by 
\begin{equation}
\mu_G(g) := \det(\Ad_g) ,
\end{equation}
where $\Ad_g:\frk g\to\frk g$ is the adjoint representation of $G$.
More explicitly, $\Ad_g = DC_g$ where $C_g(x) = gxg^{-1}$ is the conjugation by $g$.
The defining property of the modular function is that, for all $p\in G$,
\begin{equation}\label{eq6762d704}
	C_p^*\did\vol_G = \mu_G(p) \did\vol_G ,
	\qquad\text{and}\qquad
	R_{p^{-1}}^*\did\vol_G = \mu_G(p) \did\vol_G .
\end{equation}
See~\cite[Chapter 11]{MR1681462} for further details.

We write
\begin{equation}
	\kappa_G:=\tr\circ\ad\in\frk g^*,
\end{equation}
that is, $\kappa_G(v)=\tr(\ad_v)$  for all $v\in\frk g$.
Since $\kappa_G$ is linear, there exists a unique vector $\mG\in V(G)$, which we call the \emph{modular vector} of $G$, such that
\begin{equation}\label{hh:eq:mG}
	\langle \mG , v\rangle_G = \kappa_G(v)
	\quad\forall v\in V(G).
\end{equation}
If $X_1,\dots,X_m$ is an orthonormal basis of $V(G)$, then
\begin{equation}\label{eq6a7ebdf0}
	\mG = \sum_{i=1}^m\kappa_G(X_i) X_i .
\end{equation}

\begin{lemma}\label{hh:lem:modular}
Let $G$ be a sub-Riemannian Lie group with polarization $V(G)\subset\frk g$, left Haar measure $\vol_G$, modular function $\mu_G$ and modular vector $\mG$.
Let $X_1,\dots,X_m$ be an orthonormal basis of $V(G)$.
\begin{enumerate}[leftmargin=*,label=(\alph*)]
\item\label{hh:it:mod1}
	$\log\mu_G:G\to(\R,+)$ is a Lie group morphism and
	\begin{equation}\label{eq6a8d1959}
		\Diff(\log\mu_G) = \kappa_G ,
	\end{equation}
	i.e., $\tilde v(\log\mu_G)\equiv\kappa_G(v)$ for every $v\in\frk g$.
	In particular,
	\begin{equation}\label{eq6a8d1982}
		\grad_G\log\mu_G=\mG \in V(G).
	\end{equation}
\item\label{hh:it:mod2}
	For every $v\in\frk g$ we have $\div_{\vol_G}\tilde v=-\kappa_G(v)$.
	Consequently, for every $\phi\in C^\infty_c(G)$ and every $v\in\frk g$,
	\begin{equation}\label{eq6a7f1a2b}
		\int_G \tilde v \phi\did\vol_G = \kappa_G(v)\int_G \phi \did\vol_G .
	\end{equation}
\item\label{hh:it:mod3}
	For every open set $\Omega\subset G$ and every $u\in C^2(\Omega)$,
	\begin{equation}\label{hh:eq:laplacian}
		\laplacian_G u=\sum_{i=1}^m\tilde X_i^2u-\widetilde{\mG}u .
	\end{equation}
\end{enumerate}
\end{lemma}

\begin{proof}
\ref{hh:it:mod1}
From $\Ad_{gh}=\Ad_g\Ad_h$ we get
\begin{equation}\label{eq6a8d179c}
	\mu_G(gh)=\mu_G(g)\mu_G(h)
	\qquad\forall g,h\in G .
\end{equation}
Using the Jacobi formula, we get
\begin{equation}\label{eq6a8d17a7}
\begin{aligned}
	\log\mu_G(g\exp(tv))
	&=\log\mu_G(g)+\log\det\left(e^{t\ad_v}\right) \\
	&=\log\mu_G(g)+\log\left(e^{t\tr(\ad_v})\right)
	=\log\mu_G(g)+t\,\kappa_G(v) .
\end{aligned}
\end{equation}
Differentiating at $t=0$ gives $\Diff(\log\mu_G)[v] = \tilde v(\log\mu_G)=\kappa_G(v)$,
which is~\eqref{eq6a8d1959}.
Then~\eqref{hh:eq:mG} is the definition~\eqref{eq6904c16c} of the horizontal gradient,
and so we get~\eqref{eq6a8d1982}.

\ref{hh:it:mod2}
From~\eqref{eq6762d704} and~\eqref{eq6a8d1959} we get $R_a^*\did\vol_G=\mu_G(a)^{-1}\did\vol_G$ for every $a\in G$.
The flow of $\tilde v$ is $\Phi_t=R_{\exp(tv)}$, hence
\begin{align}
	(\div_{\vol_G}\tilde v)\did\vol_G
	&\overset{\eqref{eq6a8d1a70}}=\left.\frac{\diff}{\diff t}\right|_{t=0}\Phi_t^*\did\vol_G \\
	&=\left.\frac{\diff}{\diff t}\right|_{t=0}\mu_G(\exp(tv))^{-1}\did\vol_G
	=-\kappa_G(v)\did\vol_G .
\end{align}
Identity~\eqref{eq6a7f1a2b} follows from~\eqref{eq6904c1b2}.

\ref{hh:it:mod3}
Since $\tilde\grad_Gu=\sum_i(\tilde X_iu)\tilde X_i$ and~\eqref{eq670d3315},
part~\ref{hh:it:mod2} gives
$\laplacian_Gu=\sum_i\tilde X_i^2u-\sum_i\kappa_G(X_i)\tilde X_iu$,
which is~\eqref{hh:eq:laplacian} by~\eqref{eq6a7ebdf0}.
\end{proof}

\begin{remark}\label{hh:rem:axb}
As an example, consider the affine group $G=\{(a,b):a>0\}$ with $(a,b)(a',b')=(aa',b+ab')$ and $V(G)=\frk g$:
here $\tilde X=a\de_a$ and $\tilde Y=a\de_b$ are orthonormal left-invariant fields, $[X,Y]=Y$, hence $\mG=X$ and
\begin{equation}
	\laplacian_G=(a\de_a)^2+(a\de_b)^2-a\de_a=a^2(\de_a^2+\de_b^2)
\end{equation}
is the Laplace--Beltrami operator of the hyperbolic plane, as it must be.
\end{remark}
 
\subsection{Contact maps}\label{subs6a7ee050} 
Let $G$ and $H$ be polarized Lie groups and $\Omega\subset G$ open.
A $C^1$-map $F:\Omega\to H$ is a \emph{contact map}, or a \emph{horizontal map}, if
\begin{equation}
\Diff F(p)[V(G)] \subseteq V(H),
\qquad\forall p\in\Omega.
\end{equation}

Let $G$ and $H$ be sub-Riemannian Lie groups, $\Omega\subset G$ open, $X_1,\dots,X_m$ an orthonormal basis of $V(G)$, and $F:\Omega\to H$ a $C^1$ contact map.
We shall use the functions
\begin{equation}\label{hh:eq:Ai}
	A_i:\Omega\to V(H)\subset\frk h ,
	\qquad
	A_i(p):=\Diff F(p)[X_i],
	\qquad i=1,\dots,m .
\end{equation}
If $F$ is a $C^2$ contact map, then $\Diff^2F(p)|_{V(G)}$ is valued in $V(H)$
and the \emph{trace} of the bilinear map $\Diff^2F(p)|_{V(G)}:V(G)\times V(G)\to V(H)$ is
\begin{equation}
\tr_G(\Diff^2F(p)) := \sum_{i=1}^m  \Diff^2F(p)[X_i,X_i] ,
\end{equation}
for one (thus every) orthonormal basis $X_1,\dots,X_m$ of $V(G)$.
By~\eqref{eq67684517}, we have $\tilde X_iA_i=\Diff^2F[X_i,X_i]$, and thus
\begin{equation}\label{eq6a7ee065}
	\tr_G(\Diff^2F) = \sum_{i=1}^m \tilde X_iA_i .
\end{equation}

\subsection{Homothetic projections}
Let $V,W$ be Hilbert spaces.
The \emph{transpose} of a linear map $L:V\to W$ is the linear map $L^T:W\to V$ such that
\begin{equation}
\langle L^Tw,v \rangle_V = \langle w,Lv \rangle_W
\end{equation}
for all $v\in V$ and $w\in W$.

A linear map $L:V\to W$ is a \emph{homothetic projection} of factor $\lambda\ge0$ if
\begin{equation}\label{eq6708401d}
\langle L^Tw_1,L^Tw_2 \rangle_V  = \lambda^2 \langle w_1,w_2 \rangle_W ,
\end{equation}
for all $w_1,w_2\in W$,
i.e., the transpose $L^T:W\to V$ is a \emph{homothetic embedding} of factor $\lambda$.
It follows that, if $\lambda=0$ then $L=0$, while if $\lambda>0$ then $L$ is surjective.


\begin{lemma}\label{hh:lem:tight}
	Let $V$ and $W$ be Hilbert spaces with $X_1,\dots,X_m$ an orthonormal basis of $V$, $L:V\to W$ linear, and $\lambda\in[0,+\infty)$.
	Then $L$ 
%
%
	is a homothetic projection of factor $\lambda$ if and only if
	\begin{equation}\label{hh:eq:tight}
		\sum_{i=1}^m\langle w,LX_i\rangle_W \cdot LX_i=\lambda^2\,w
		\quad\forall w\in W .
	\end{equation}
	In this case, for every linear map $P:W\to W$,
	\begin{equation}\label{hh:eq:tighttrace}
		\sum_{i=1}^m\langle PLX_i,LX_i\rangle_W = \lambda^2 \tr(P).
	\end{equation}
\end{lemma}

\begin{proof}
For $w\in W$ we have $L^Tw=\sum_i\langle L^Tw,X_i\rangle_V X_i=\sum_i\langle w,LX_i\rangle_W X_i$, hence
\begin{equation}
	\langle L^Tw_1,L^Tw_2\rangle_V
	=\sum_{i=1}^m\langle w_1,LX_i\rangle_W \langle w_2,LX_i\rangle_W ,
\end{equation}
and by~\eqref{eq6708401d} this equals $\lambda^2\langle w_1,w_2\rangle_W$ for all $w_1,w_2\in W$ if and only if~\eqref{hh:eq:tight} holds.
For~\eqref{hh:eq:tighttrace}, expand the trace of $P\circ\sum_iLX_i\otimes LX_i=\lambda^2P$ in an orthonormal basis $Y_1,\dots,Y_n$ of $W$:
\begin{equation}
	\lambda^2\tr(P)
	=\sum_{a=1}^n\Big\langle P\sum_{i=1}^m\langle Y_a,LX_i\rangle_W LX_i,Y_a\Big\rangle_H
	=\sum_{i=1}^m\langle PLX_i,LX_i\rangle_H . 
\end{equation}
\end{proof}

\subsection{Conformal submersion of sub-Riemannian Lie groups}\label{sec670d0abe}

Let $G$ and $H$ be sub-Riemannian Lie groups, and $\Omega\subset G$ open. 
A \emph{$C^1$-conformal submersion} of factor $\lambda:\Omega\to[0,+\infty)$ is a $C^1$ contact map $F:\Omega\to H$,
such that, for every $p\in\Omega$, the restriction $\Diff F(p)|_{V(G)}:V(G)\to V(H)$ is a linear homothetic projection of factor $\lambda(p)$.
Conformal submersions are also known as \emph{semiconformal maps}.

We stress that the term ``submersion'' is slightly abused since $\lambda(p)$ may be zero, in which case $\Diff F(p)|_{V(G)}=0$.
However, {\it if $F$ is a $C^\infty$-smooth conformal submersion of factor $\lambda$, then $F$ is a submersion on $\{\lambda\neq0\}$}.
Indeed, if $\lambda(p)\neq0$, then $\Diff F(p)[V(G)]=V(H)$.
Since $V(H)$ is bracket-generating, $F$ is a smooth submersion on $\{\lambda\neq0\}$, that is, the Lie differential $\Diff F(p):\frk g\to\frk h$ is surjective for all $p\in\Omega\cap\{\lambda\neq0\}$, see Lemma~\ref{lem6a7d950c}.

Let $X_1,\dots,X_m$ be an orthonormal basis of $V(G)$, $A_1,\dots,A_m$ as in~\eqref{hh:eq:Ai}, and $Y_1,\dots,Y_n$ an orthonormal basis of $V(H)$.
Taking $w=Y_a$ in~\eqref{hh:eq:tight} and summing over $a$, we see that the factor $\lambda$ of a conformal submersion $F$ is determined by $F$, because
\begin{equation}\label{eq6a7ee0a1}
	\sum_{i=1}^m|A_i|_H^2 = n\,\lambda^2 .
\end{equation}
%

\subsection{Horizontal energy and harmonic maps}\label{subs6a7ee0c8}
Let $G$ and $H$ be sub-Riemannian Lie groups, $\Omega\subset G$ open, and $X_1,\dots,X_m$ an orthonormal basis of $V(G)$.
The \emph{horizontal energy} of a $C^1$ contact map $F:\Omega\to H$ is
\begin{equation}\label{hh:eq:energy}
	\Energy(F):=\frac12\int_\Omega\sum_{i=1}^m\bigl|\Diff F[X_i]\bigr|_H^2\did\vol_G
	\ \in[0,+\infty] ,
\end{equation}
which does not depend on the choice of the orthonormal basis.

%
%

A \emph{contact variation} of a smooth contact map $F:\Omega\to H$ is a smooth family $(F_s)_{s\in(-\epsilon,\epsilon)}$ of contact maps $\Omega\to H$ with $F_0=F$ and such that $F_s=F$ on $\Omega\setminus K$, for every $s$, for some compact set $K\subset\Omega$;
the \emph{variation field} of $(F_s)_s$ is
\begin{equation}\label{hh:eq:variationfield}
	\varphi\in C^\infty_c(\Omega;\frk h),
	\qquad
	\varphi(p):=\left.\frac{\diff}{\diff s}\right|_{s=0}F(p)^{-1}F_s(p) ,
	\qquad\forall p\in\Omega.
\end{equation}
Since $F_s=F$ outside $K$, the derivative at $s=0$ of $s\mapsto\frac12\int_K\sum_i|\Diff F_s[X_i]|_H^2\did\vol_G$ does not depend on the choice of $K$:
we denote it by $\left.\frac{\diff}{\diff s}\right|_{s=0}\Energy(F_s)$, also when $\Energy(F)=+\infty$.
The map $F$ is a \emph{harmonic map} if
\begin{equation}
	\left.\frac{\diff}{\diff s}\right|_{s=0}\Energy(F_s) = 0
	\qquad\text{for every variation $(F_s)_s$ of $F$.}
\end{equation}

We stress that every $F_s$ is required to be a contact map.

\section{Regularity of harmonic morphisms}\label{sec6a7b20bc}

\subsection{Smoothness of horizontal harmonic morphisms}\label{subs6a7b21a4}

In the following theorem, we show that harmonic morphisms are smooth if we assume continuity and a weak condition on the preservation of the horizontal structures.
We stress that the definition of harmonic morphism does not assume the map to be differentiable or to preserve the horizontal distributions.

A sub-Riemannian manifold $M$ with horizontal bundle $\tilde V\subset TM$ is \emph{equiregular} if, 
for every $k\in\N$, the set $\tilde V^{(k)}\subset TM$ is a subbundle of $TM$, where $\tilde V^{(k)}$ is the set generated  $k$-iterated brackets of sections of $\tilde V$.
We use equiregularity only in the following Theorem~\ref{thm6a7ee1f4}, because it is required in \cite[Theorem 4.4]{MR3912638}.

\begin{theorem}[Theorem~\ref{thm69959b8b}]\label{thm6a7ee1f4}
	Let $M$ and $N$ be measured equiregular sub-Riemannian manifolds and let $F:M\to N$ be a continuous harmonic morphism.
	Suppose that $F$ maps horizontal curves to horizontal curves,
	that is, if $\gamma$ is an AC horizontal curve in $M$, then $F\circ\gamma$ is an AC horizontal curve in $N$.
	Then $F$ is $C^\infty$-smooth.
\end{theorem}
\begin{proof}
	The proof is a direct application of results from~\cite{MR3912638}.
	
	Fix $p\in M$ and $q=F(p)\in N$.
	First, by \cite[Theorem 4.4]{MR3912638}, there is a system of coordinates $y_1,\dots,y_n$ for $N$ at $q$ so that $y_1,\dots,y_r$ are $\laplacian_N$-harmonic functions, where $r$ is the rank of $N$.
	Since $F$ is harmonic, the functions $F_k = y_k\circ F$ are $\laplacian_M$-harmonic for $k\in\{1,\dots,r\}$, and thus smooth.
	Since $F$ maps horizontal curves to horizontal curves,
	we can apply~\cite[Proposition 4.14]{MR3912638} to obtain that $F$ is $C^\infty$-smooth.
\end{proof}

\subsection{Harmonic coordinates}\label{subs6a7b219a}

Recall from Section~\ref{subs6a7eff81} that a measured sub-Rie\-man\-nian manifold is harmonically homogeneous if it carries a transitive Lie group action by harmonic morphisms, and that sub-Riemannian Lie groups are harmonically homogeneous.

\begin{theorem}[Theorem~\ref{thm6a7b21cf}]\label{thm6a7ee20d}
	Let $M$ be a measured sub-Riemannian manifold that is harmonically homogeneous.
	Then every point of $M$ has harmonic coordinates, that is, the coordinate functions are harmonic functions.
\end{theorem}
\begin{proof}
	Let $K$ be a Lie group and let $K\times M\to M$ be a transitive action by harmonic morphisms of $M$.
	For every $p\in M$, define
	\begin{equation}
		\tilde W_p := \{v\in T_pM : \diff u(p)[v] = 0 \text{ for every harmonic function $u$ defined near $p$}\} .
	\end{equation}
	If $k\in K$ and $u$ is harmonic near $kp$, then $u\circ k$ is harmonic near $p$, because $k$ acts as a harmonic morphism;
	hence $\diff u(kp)[\diff k(p)[v]] = \diff (u\circ k)(p)[v] = 0$ for every $v\in \tilde W_p$, that is, $\diff k(p)[\tilde W_p]\subseteq \tilde W_{kp}$.
	Applying the same argument to $k^{-1}$, we get $\diff k(p)[\tilde W_p]= \tilde W_{kp}$.
	Since the action is transitive, the set $\tilde W:=\bigsqcup_{p\in M}\tilde W_p \subset TM$ is a $K$-invariant smooth vector subbundle of $TM$.

	By construction, for every $X\in\Gamma(\tilde W)$, we have that $Xu=0$ for every harmonic function $u$ on $M$.

	Arguing by contradiction, assume that $\tilde W\neq0$, and thus there is a non-zero smooth vector field $X\in\Gamma(\tilde W)$.
	Fix $p\in M$ with $X(p)\neq0$ and $\gamma:(-\epsilon,\epsilon)\to M$ with $\gamma(0)=p$ and $\dot\gamma(t) = X(\gamma(t))$ for every $t\in(-\epsilon,\epsilon)$, for some $\epsilon>0$.

	Let $f$ be a smooth map defined on some open subset $\Omega$ of $M$ such that $Xf\equiv 1$ and $p\in\Omega$.

	Since measured sub-Riemannian manifolds are Brelot harmonic spaces, see Proposition~\ref{prop69947463}, we can find a $\laplacian_M$-regular domain $U\Subset\Omega$ with $p\in U$ and $\gamma((-\epsilon,\epsilon))\cap M\setminus U \neq \emptyset$.
	In particular, since $\gamma^{-1}(U)$ is an open subset of $\R$ containing $0$, there are $t_-<0$ and $t_+>0$ such that $\gamma(t_-)\in\de U$, $\gamma(t_+)\in\de U$, and $\gamma(t) \in U$ for every $t\in (t_-,t_+)$.

	Let $u\in C(\bar U)$ be a harmonic function in $U$ and so that $u=f$ on $\de U$.
	Since $u$ is harmonic, then $Xu=0$ and thus $u(\gamma(t_-)) = u(\gamma(t_+))$.
	Since $Xf=1$, then $f(\gamma(t_+)) - f(\gamma(t_-)) = t_+ - t_- >0$:
	contradiction.

	We conclude that $\tilde W=0$ and thus, for every $p\in M$, there are harmonic functions $u_1,\dots,u_n$ so that $\diff u_1(p),\dots,\diff u_n(p)$ is a basis of $T_p^*M$, and thus these functions form a system of coordinates at $p$.
\end{proof}

\begin{theorem}[Theorem~\ref{thm6a7b21d7}]\label{thm6a7ee21f}
	Let $M$ and $N$ be measured sub-Riemannian manifolds and $F:M\to N$ a harmonic morphism.
	Suppose that $N$ is harmonically homogeneous.
	Then $F$ is $C^\infty$-smooth.
\end{theorem}
\begin{proof}
	Fix $p\in M$.
	By Theorem~\ref{thm6a7ee20d}, there are an open neighborhood $\Omega_N$ of $F(p)$ in $N$ and harmonic functions $u_1,\dots,u_n$ on $\Omega_N$ that form a system of coordinates on $\Omega_N$.
	Since $F$ is a harmonic morphism, the functions $u_k\circ F$ are harmonic on the open set $F^{-1}(\Omega_N)$, and thus they are smooth, by the hypoellipticity of $\laplacian_M$.
	Since $(u_1,\dots,u_n)$ is a chart of $N$, we conclude that $F$ is smooth on $F^{-1}(\Omega_N)$.
\end{proof}

Since sub-Riemannian Lie groups are harmonically homogeneous, we obtain the following.

\begin{corollary}\label{cor6995a31c}
	Let $G$ and $H$ be sub-Riemannian Lie groups, $\Omega\subset G$ open, and $F:\Omega\to H$ a harmonic morphism.
	Then $F$ is $C^\infty$-smooth.
\end{corollary}

\section{Proof of Theorem~\ref{thm699465f3}}\label{sec699472d5}
\newcommand{\coeff}{f}

\subsection{Pull back of differential operators}

A part of classical results for harmonic spaces,
Proposition~\ref{prop68fbbcbb}, which follows from Lemma~\ref{lem68fbae06}, covers the main step of the proof of Theorem~\ref{thm699465f3}.

\begin{lemma}\label{lem68fbae06}
	Let $V$ be a vector space over a field $\K$ and $\alpha,\beta\in\Lin_\K(V;\K)$.
	If $\ker(\alpha)\subset\ker(\beta)$, then there exists $\coeff\in \K$ such that
	\begin{equation}\label{eq68fbb9de}
		\beta v = \coeff \cdot \alpha v ,
		\qquad\forall v\in V .
	\end{equation}
\end{lemma}
\begin{proof}
	If $\alpha=0$, then $\beta=0$ as well and we can take $\coeff=0$.
	Suppose $\alpha\neq0$: then there is $\hat v\in V$ such that $\alpha \hat v\neq0$.
	If $v\in V$, then 
	$\alpha\left( (\alpha\hat v)v-(\alpha v)\hat v \right) = 0$ and the assumption $\ker(\alpha)\subset\ker(\beta)$ implies that
	$\beta\left( (\alpha\hat v)v-(\alpha v)\hat v \right) = 0$. It follows by linearity that $(\alpha\hat v)(\beta v)= (\alpha v)(\beta\hat v)$ 	which proves~\eqref{eq68fbb9de} with $\coeff= \beta\hat v / \alpha\hat v$.
\end{proof}

\begin{proposition}\label{prop68fbbcbb}
	Let $M,N$ be real manifolds and $F:M\to N$ a map of class $C^k$ for some $k\in\N$.
	Let $P_M$ and $P_N$ be differential operators of class $C^p$ and order $d$ on $M$ and $N$, respectively, with $p,d\le k$. 
	Assume that, for every $y\in N$, $P_N|_y\neq0$.
	Assume also that for every open set $U\subset N$ and $u\in C^d(U;\R)$, if $P_Nu\ge0$ on $U$, then $P_M(u\circ F)\ge0$ on $F^{-1}(U)$.
	
	Then there exists a function $\coeff:M\to [0,+\infty)$ of class $C^s$ with $s=\min\{k-d,p\}$ such that 
	\begin{equation}\label{eq68fbbe07}
		P_M\circ F^* = \coeff \cdot (F^*\circ P_N) .
	\end{equation}
	Equivalently, for every $U\subset N$ open and $u\in C^k(U)$,
	\begin{equation}
	P_M(u\circ F) = \coeff \cdot (P_Nu)\circ F
    \qquad\text{on }F^{-1}(U).
	\end{equation}
\end{proposition}
\begin{proof}
	Fix $x\in M$, put $y=Fx$, and let $C^k_y$ denote the set of germs of $C^k$ functions at $y$.
	Consider the two linear functionals $\alpha,\beta:C^k_y\to\R$ defined by:
	$\beta u = (P_M(u\circ F))x$ and $\alpha u = (P_Nu)y$.
	
	We claim that 
	\begin{equation}\label{eq6904f0e7}
		\ker(\alpha) \subset\ker(\beta) .
	\end{equation}
	Indeed, since $P_N|_y\neq0$, then there exists $w\in C^k_y$ such that $P_Nw(y)>0$.
	If $u\in C^k_y$ satisfies $\alpha u = 0$, then for every $\epsilon>0$, there is a  neighborhood of $y$ where $P_N(u+\epsilon w) = P_Nu + \epsilon P_Nw >0$. 
	It follows that for every $\epsilon>0$ there is a  neighborhood of $x$ where $P_M((u+ \epsilon w)\circ F)\ge0$ and
	we conclude that $P_M(u\circ F)x\ge0$.
	Similarly, for every $\epsilon>0$ there is a neighborhood of $y$ where $P_N(u-\epsilon w) = P_Nu - \epsilon P_Nw < 0$. 
	Hence for every $\epsilon>0$ there is a  neighborhood of $x$ where $P_M((u- \epsilon w)\circ F)\le 0$ and
	we conclude that $P_M(u\circ F)x\le0$.
	It follows that $\beta u =P_M(u\circ F)x=0$ which proves the claim at \eqref{eq6904f0e7}.
	
    By Lemma~\ref{lem68fbae06}, there exists $\coeff\in\R$ such that 
	$\beta=\coeff\alpha$.
	So, 
	there exists a function $\coeff:M\to\R$ such that, for every $U\subset N$ open and every $u\in C^k(U;\R)$, 
	$P_M(u\circ F) = \coeff \cdot ((P_Nu)\circ F) $,
	which proves \eqref{eq68fbbe07}.
	
	To check the regularity of $\coeff$ we begin by observing that if $x\in M$ then $P_N|_{F(x)}\neq0$,
	and so there is  a neighborhood $V$ of $F(x)$ in $N$ and $u\in C^k(V;\R)$ such that
	$P_Nu(y)>0$ for all $y\in V$.
	Therefore, $((P_Nu)\circ F)(z)>0$ for all $z\in F^{-1}(V)$, where $F^{-1}(V)$ is an open neighborhood of $x$ since $F$ is continuous.
	Identity~\eqref{eq68fbbe07} implies that, 
	\begin{equation}\label{eq6901eb82}
		\forall z\in F^{-1}(V),\qquad
		\coeff(z) = \frac{P_M(u\circ F)(z)}{ (P_Nu)\circ F (z)} \ge 0 .
	\end{equation}
	Set $s = \min\{k-d,p\}$.
	Notice that $P_Nu\in C^s(V;\R)$ and thus $(P_Nu)\circ F\in C^s(F^{-1}(V);\R)$.
	Moreover, $u\circ F\in C^k(F^{-1}(V);\R)$ and thus \[P_M(u\circ F) \in C^s(F^{-1}(V);\R).\]
	We conclude that $\coeff\in C^s(F^{-1}(V);\R)$.
\end{proof}

\subsection{Proof of Theorem~\ref{thm699465f3}}

	It is clear that $\ref{thm699465f3_item4}\THEN\ref{thm699465f3_item1}$.
	We need to show the reverse implication $\ref{thm699465f3_item1}\THEN\ref{thm699465f3_item4}$.
	
	First of all, we note that $M$ and $N$ are Brelot harmonic spaces by Proposition~\ref{prop69947463}.
	If $F$ is a harmonic morphism in the sense of Section~\ref{subs699474ca}, 
	then Proposition~\ref{prop6947b0c3} implies that $F$ pulls back superharmonic functions from $N$ to $M$.
	Proposition~\ref{prop69947448} ensures that the hypothesis of Proposition~\ref{prop68fbbcbb} are satisfied and thus we get~\eqref{eq699465b3} from~\eqref{eq68fbbe07}.
	
	Proposition~\ref{prop68fbbcbb} implies also that $f\ge0$ and that $f$ is of class $C^{k-2}$, which concludes the proof of Theorem~\ref{thm699465f3}.
\qed

\section{Harmonic morphisms of sub-Riemannian Lie groups}\label{sec6a7ee0f3}

In this section we prove Theorem~\ref{thm6a7ae38c}.
The main ingredients of the proof are Theorem~\ref{thm699465f3} and \cite[Theorem A]{2025arXiv250100576K}.
We will restate the latter in Proposition~\ref{prop6763089e} using the harmonic morphism operator defined in Section~\ref{subs6a7ee082}.

\subsection{The harmonic morphism operator}\label{subs6a7ee082}
Let $G$ and $H$ be sub-Riemannian Lie groups, $\Omega\subset G$ open, and $F:\Omega\to H$ a $C^2$ conformal submersion of factor $\lambda$.
The \emph{harmonic morphism operator} of $F$ is the function
\begin{equation}\label{hh:eq:T}
	\Tt(F):=\tr_G(\Diff^2F)-\Diff F[\mG]+\lambda^2\,\mH
	\ :\ \Omega\to V(H) ,
\end{equation}
where $\mG\in V(G)$ and $\mH\in V(H)$ are the modular vectors of $G$ and $H$.
By~\eqref{eq6a7ee0a1}, the operator $\Tt(F)$ depends only on $F$.
Equation~\eqref{eq6a7c3003} of Theorem~\ref{thm6a7ae38c} reads as $\Tt(F)=0$.

\subsection{Symmetries of the sub-Riemannian Laplacian}\label{subs6a8df50d}

We recall a result from~\cite{2025arXiv250100576K} and show a consequence we need for harmonic morphisms in Corollary~\ref{cor6a7ec445}.

\begin{proposition}[{\cite[Theorem A]{2025arXiv250100576K}}]\label{prop6763089e}
	Let $G$ and $H$ be sub-Riemannian Lie groups, $\Omega_G\subset G$ and $\Omega_H\subset H$ open,
	$F:\Omega_G\to \Omega_H$ a $C^2$-smooth map,
	and $\lambda:\Omega_G\to[0,+\infty)$, $b:\Omega_G\to V(H)$ and $c:\Omega_G\to\R$ continuous functions.
	The following statements are equivalent:
	\begin{enumerate}[label=(\roman*)]
	\item\label{thm6763089e_1}
	for every $v\in C^2(\Omega_H)$,
	\begin{equation}\label{eq6a7ee3a2}
		\laplacian_G(v\circ F) = \lambda^2 \cdot (\laplacian_Hv)\circ F
		+ \langle b , (\grad_Hv)\circ F \rangle_H + c \cdot (v\circ F) ;
	\end{equation}
	\item\label{thm6763089e_2}
	$F$ is a conformal submersion of factor $\lambda$, $c\equiv0$ and $b = \Tt(F)$.
	\end{enumerate}
\end{proposition}

\begin{corollary}[Composition formula]\label{cor6a7ec445}
	Let $G$ and $H$ be sub-Riemannian Lie groups, $\Omega\subset G$ open,
	and $F:\Omega\to H$ a $C^2$ conformal submersion of factor $\lambda:\Omega\to[0,+\infty)$.
	Then, for every every $v\in C^2(H)$ and every $p\in\Omega$,
	\begin{equation}\label{hh:eq:composition}
		\laplacian_G(v\circ F)(p)
		=\lambda(p)^2\,(\laplacian_Hv)(F(p))
		+\Diff v(F(p))\bigl[\Tt(F)(p)\bigr] .
	\end{equation}
	Consequently, $F$ is a harmonic morphism if and only if $\Tt(F)\equiv0$.
\end{corollary}
%
\begin{proof}
	Apply Proposition~\ref{prop6763089e} to $F$ with $b:=\Tt(F)$, which is a continuous $V(H)$-valued function, and with $c\equiv0$:
	statement~\ref{thm6763089e_2} holds by assumption, and thus~\eqref{eq6a7ee3a2} holds.
	Since $\Tt(F)(p)\in V(H)$, the definition of the horizontal gradient gives
	$\langle \Tt(F)(p),\grad_Hv(F(p))\rangle_H = \Diff v(F(p))[\Tt(F)(p)]$,
	and~\eqref{hh:eq:composition} follows.

	If $\Tt(F)\equiv0$, then~\eqref{hh:eq:composition} shows that $v\circ F$ is harmonic whenever $v$ is harmonic, that is, $F$ is a harmonic morphism.
	Conversely, suppose that $F$ is a harmonic morphism.
	By Theorem~\ref{thm699465f3}, there is $\coeff:\Omega\to[0,+\infty)$ such that $\laplacian_G(v\circ F) = \coeff\cdot(\laplacian_Hv)\circ F$ for every $v\in C^2(H)$;
	the function $\coeff$ is continuous, and so is $\sqrt{\coeff}$.
	Hence statement~\ref{thm6763089e_1} of Proposition~\ref{prop6763089e} holds with $\sqrt{\coeff}$, $b=0$ and $c=0$ in place of $\lambda$, $b$ and $c$.
	Statement~\ref{thm6763089e_2} then implies that $F$ is a conformal submersion of factor $\sqrt f$ and that $\Tt(F)=0$.
	By~\eqref{eq6a7ee0a1}, the factor of a conformal submersion is unique, so that $\sqrt{\coeff}=\lambda$.
\end{proof}

\subsection{Proof of Theorem~\ref{thm6a7ae38c}}

	\ref{thm6a7ae38c_item1}$\THEN$\ref{thm6a7ae38c_item2}.
	By Corollary~\ref{cor6995a31c}, the map $F$ is $C^\infty$-smooth.
	Theorem~\ref{thm699465f3} then gives a function $\coeff:\Omega\to[0,+\infty)$ such that $\laplacian_G(v\circ F) = \coeff\cdot(\laplacian_Hv)\circ F$ for every $v\in C^\infty(H)$.
	Set $\lambda:=\sqrt{\coeff}$.

	\ref{thm6a7ae38c_item2}$\THEN$\ref{thm6a7ae38c_item3}.
	Both sides of~\eqref{eq6a7ae3d9} at a point $p\in\Omega$ depend only on the $2$-jet of $v$ at $F(p)$, and every $2$-jet at a point of $H$ is the $2$-jet of a function in $C^\infty(H)$.
	Therefore, statement~\ref{thm6a7ae38c_item2} implies statement~\ref{thm6763089e_1} of Proposition~\ref{prop6763089e} with $b=0$ and $c=0$, for every $\Omega_H\subset H$ open.
	The function $\lambda$ is continuous because $\lambda^2 = \laplacian_G(v\circ F)/(\laplacian_Hv)\circ F$ for any $v\in C^\infty(H)$ with $\laplacian_Hv>0$ near $F(p)$.
	Hence statement~\ref{thm6763089e_2} of Proposition~\ref{prop6763089e} holds, that is, $F$ is a conformal submersion of factor $\lambda$ and $\Tt(F)=0$, which is~\eqref{eq6a7c3003}.

	\ref{thm6a7ae38c_item3}$\THEN$\ref{thm6a7ae38c_item1}.
	This is the last statement of Corollary~\ref{cor6a7ec445}.
\qed

\section{Harmonicity of harmonic morphisms: the first variation for harmonic morphisms}\label{hh:sec}

The theorem of Fuglede~\cite{zbMATH03530589} and Ishihara~\cite{MR545705} states that a map between Riemannian manifolds is a harmonic morphism if and only if it is a conformal submersion (a.k.a.~semiconformal maps, see Section~\ref{sec670d0abe}) and a harmonic map.
In the following three sections we discuss the corresponding question in the setting of sub-Riemannian Lie groups:
\emph{is a harmonic morphism a harmonic map, i.e., a critical point of the horizontal energy~\eqref{hh:eq:energy}?}
Conjecturally, is the equation $\Tt(F)=0$ in Theorem~\ref{thm6a7ae38c} the Euler--Lagrange equation of the horizontal energy among contact maps?
Such Euler--Lagrange equations have been obtained by Grong and Markina in~\cite[Theorem 4.2]{zbMATH07887823}.

%

The answer is negative, and the obstruction is a linear form on the Lie algebra of the target alone, the modular mismatch $\chi$ of Section~\ref{hh:subs:chi}.
The crucial hinge to this section is the pointwise identity of Proposition~\ref{hh:prop:key}, which we prove in Section~\ref{hh:subs:variation}:
it computes the first variation of the horizontal energy along a conformal submersion in terms of $\Tt(F)$ and of $\chi$.
From it we deduce that harmonic morphisms are harmonic maps whenever $\chi=0$ for some complement of the polarization of the target (Section~\ref{hh:subs:positive}), which is the case for Riemannian targets and for Carnot targets, and that this fails in general (Section~\ref{hh:subs:counterexample}).
In Section~\ref{hh:subs:GM} we compare our identity with the Euler--Lagrange equations of~\cite{zbMATH07887823}.

\begin{remark}\label{hh:rem:why}
The reason for the discrepancy is already visible in the definitions.
The horizontal energy~\eqref{hh:eq:energy} does not see the measure of the target: only the scalar product on $V(H)$ enters it.
The harmonic morphism equation~\eqref{eq6a7c3003}, on the contrary, sees $\vol_H$ through the modular vector $\mH$, because $\laplacian_H$ does.
The two agree in particular when the modular data of $H$ are compatible with a splitting $\frk h=\Ee\oplus\Qq$, which is exactly the condition $\chi=0$.
In the Riemannian case there is no splitting to be compatible with;
in the Carnot case the group is unimodular \emph{and} the polarization is the bottom layer of a grading.
Both times the mismatch is invisible.
\end{remark}

\subsection{The modular mismatch}\label{hh:subs:chi}

\begin{definition}\label{hh:def:chi}
Let $H$ be a polarized Lie group with polarization $\Ee\subset\frk h$
and let $\Qq\subset\frk h$ be a complement of $\Ee$ in $\frk h$, that is,
\begin{equation}\label{hh:eq:splitting}
	\frk h=\Ee\oplus\Qq ,
\end{equation}
with projections $\pi_{\Ee}:\frk h\to\Ee$ and $\pi_{\Qq}:\frk h\to\Qq$.
The \emph{horizontal} and the \emph{vertical modular characters} of the splitting~\eqref{hh:eq:splitting} are the linear forms $\kae,\kaq\in\frk h^*$,
\begin{equation}\label{hh:eq:kappas}
	\kae(B):=\tr\bigl(\pi_{\Ee}\circ\ad_B|_{\Ee}\bigr),
	\qquad
	\kaq(B):=\tr\bigl(\pi_{\Qq}\circ\ad_B|_{\Qq}\bigr),
	\qquad \forall B\in\frk h ,
\end{equation}
so that $\kae+\kaq=\kappa_H$.
The \emph{modular mismatch} of the splitting~\eqref{hh:eq:splitting} is
\begin{equation}\label{hh:eq:chi}
	\chi:=\kappa_H\circ\pi_{\Ee}-\kae
	=\kaq\circ\pi_{\Ee}-\kae\circ\pi_{\Qq}
	\ \in\frk h^* .
\end{equation}
We say that the complement $\Qq$ is \emph{modularly balanced} if $\chi=0$, that is, if
\begin{equation}\label{hh:eq:balanced}
	\tr\bigl(\pi_{\Ee}\ad_Z|_{\Ee}\bigr)=0\quad\forall Z\in\Qq ,
	\quad\text{and}\quad
	\tr\bigl(\pi_{\Qq}\ad_B|_{\Qq}\bigr)=0\quad\forall B\in\Ee .
\end{equation}
\end{definition}

Notice that $\chi$ depends only on the polarized Lie algebra $(\frk h,\Ee)$ and on the choice of $\Qq$:
it involves neither a map $F$, nor a source group $G$, nor the scalar products.

\begin{proposition}\label{hh:prop:examples}
The following polarized Lie algebras $(\frk h,\Ee)$ admit a modularly balanced complement:
\begin{enumerate}[label=(\alph*)]
\item\label{hh:it:riem} \emph{Riemannian Lie groups}: if $\Ee=\frk h$;
\item\label{hh:it:ideal} if $\Ee$ admits a complement $\Qq$ that is an ideal of $\frk h$ with $\tr(\ad_B|_{\Qq})=0$ for every $B\in\Ee$;
\item\label{hh:it:carnot} \emph{Carnot groups}: if $\frk h$ is stratified by $\frk h=V_1\oplus\dots\oplus V_s$ and $\Ee=V_1$;
\item\label{hh:it:unimod} \emph{Unimodular Lie groups with trace-free horizontal action}: if $\kappa_H=0$ and $\kae=0$ for some complement $\Qq$.
\end{enumerate}
\end{proposition}

\begin{proof}
\ref{hh:it:riem} Take $\Qq=\{0\}$: then $\pi_{\Qq}=0$ and $\kaq=0$, so $\chi=0$.

\ref{hh:it:ideal} If $\Qq$ is an ideal and $Z\in\Qq$, then $[Z,\Ee]\subseteq\Qq$, hence $\pi_{\Ee}\ad_Z|_{\Ee}=0$ and $\kae|_{\Qq}=0$.
Moreover $\ad_B(\Qq)\subseteq\Qq$ for every $B\in\frk h$, so $\kaq(B)=\tr(\ad_B|_{\Qq})$, which vanishes on $\Ee$ by assumption.

\ref{hh:it:carnot} Take $\Qq=V_2\oplus\dots\oplus V_s$, which is an ideal; for $B\in V_1$ the map $\ad_B$ sends $V_j$ into $V_{j+1}$, so $\ad_B|_{\Qq}$ is nilpotent and traceless.  Apply~\ref{hh:it:ideal}.

\ref{hh:it:unimod} $\kaq=\kappa_H-\kae=0$, so $\chi=0$ by~\eqref{hh:eq:chi}.
%
\end{proof}

\subsection{The first variation of the horizontal energy}\label{hh:subs:variation}

We first compute the derivative of the Lie differential along a variation.
Recall from Section~\ref{subs6a7ee0c8} that all the maps of a variation are contact maps.

\begin{lemma}\label{hh:lem:firstvar}
Let $G$ and $H$ be sub-Riemannian Lie groups, $\Omega\subset G$ open, $X_1,\dots,X_m$ an orthonormal basis of $V(G)$, and let $F:\Omega\to H$ be a smooth contact map with $A_1,\dots,A_m$ as in~\eqref{hh:eq:Ai}.
Let $(F_s)_{s\in(-\epsilon,\epsilon)}$  
be a variation of $F$ with variation field $\varphi$.
Then 
\begin{equation}\label{hh:eq:dalpha}
	\left.\frac{\diff}{\diff s}\right|_{s=0}\Diff F_s[v]
	=\tilde v\varphi+[\Diff F_s[v],\varphi] ,
	\qquad\forall v\in\frk g.
\end{equation}
This implies that 
\begin{equation}\label{hh:eq:admissible}
	\tilde v\varphi+[\Diff F_s[v],\varphi]\in\Ee,
	\qquad\forall v\in V(G),
\end{equation}
and
\begin{equation}\label{hh:eq:firstvar}
	\left.\frac{\diff}{\diff s}\right|_{s=0}\Energy(F_s)
	=\int_\Omega\sum_{i=1}^m\bigl\langle A_i,\tilde X_i\varphi+[A_i,\varphi]\bigr\rangle_H\did\vol_G .
\end{equation}
\end{lemma}

\begin{proof}
Set $c_s(p):=F(p)^{-1}F_s(p)$, so that $c_0\equiv1$ and $\left.\frac{\diff}{\diff s}\right|_{s=0}c_s=\varphi$.
For $v\in\frk g$ we have
\begin{equation}
\begin{aligned}
	F_s(p)^{-1}F_s(p\exp(tv))
	&= \Bigl( c_s(p)^{-1}\bigl(F(p)^{-1}F(p\exp(tv))\bigr)c_s(p) \Bigr)\\
	&\qquad\cdot \Bigl( c_s(p)^{-1}c_s(p\exp(tv)) \Bigr) ,
\end{aligned}
\end{equation}
which is a product of two curves through the identity of $H$ whose derivatives at $t=0$ therefore add:
\begin{equation}\label{hh:eq:DFs2}
	\Diff F_s(p)[v] = \Ad\bigl(c_s(p)^{-1}\bigr)\Diff F(p)[v] + \Diff c_s(p)[v] .
\end{equation}
We differentiate~\eqref{hh:eq:DFs2} at $s=0$.
On the one hand, $\left.\frac{\diff}{\diff s}\right|_{s=0}\Ad(c_s^{-1})=-\ad_\varphi$.
On the other hand, since $c_0\equiv1$, for every fixed $t$ the derivative at $s=0$ of $s\mapsto c_s(p)^{-1}c_s(p\exp(tv))$ is $\varphi(p\exp(tv))-\varphi(p)$, hence
$\left.\frac{\diff}{\diff s}\right|_{s=0}\Diff c_s(p)[v] = \tilde v\varphi(p)$.
We have obtained~\eqref{hh:eq:dalpha}.

Each $F_s$ is a contact map, so $\Diff F_s[v]\in\Ee$ for every $v\in V(G)$ and every $s$, and~\eqref{hh:eq:admissible} follows by differentiating at $s=0$.
Finally, all the $F_s$ agree outside a fixed compact subset of $\Omega$, so we may differentiate~\eqref{hh:eq:energy} under the integral sign, and~\eqref{hh:eq:firstvar} follows from~\eqref{hh:eq:dalpha}.
\end{proof}

The following identity contains all the results of this section.

\begin{proposition}[Key identity]\label{hh:prop:key}
Let $G$ and $H$ be sub-Riemannian Lie groups, $\Omega\subset G$ open, $X_1,\dots,X_m$ an orthonormal basis of $V(G)$, and $F:\Omega\to H$ a $C^2$-smooth conformal submersion of factor $\lambda$, with $A_1,\dots,A_m$ as in~\eqref{hh:eq:Ai}.
Let $\Qq$ be a complement of $\Ee$ in $\frk h$ and let $\chi$ be the modular mismatch of the splitting $\frk h=\Ee\oplus\Qq$.
Then, for every $\varphi\in C^\infty(\Omega;\frk h)$,
\begin{equation}\label{hh:eq:key}
\begin{aligned}
	\sum_{i=1}^m\bigl\langle A_i,\pi_{\Ee}\bigl(\tilde X_i\varphi+[A_i,\varphi]\bigr)\bigr\rangle_H
	&=\div_{\vol_G}\Bigl(\sum_{i=1}^m\langle A_i,\pi_{\Ee}\varphi\rangle_H\,\tilde X_i\Bigr)\\
	&\qquad-\bigl\langle \Tt(F),\pi_{\Ee}\varphi\bigr\rangle_H
	+\lambda^2\,\chi(\varphi) ,
\end{aligned}
\end{equation}
where $\Tt(F)$ is the harmonic morphism operator~\eqref{hh:eq:T}.
\end{proposition}

\begin{proof}
We compute the two summands of the left-hand side of~\eqref{hh:eq:key}.

For the second one, observe that $P_\varphi:=\pi_{\Ee}\circ\ad_\varphi|_{\Ee}$ is an endomorphism of $\Ee$, that $\pi_{\Ee}[A_i,\varphi]=-P_\varphi A_i$, and that $A_i\in\Ee$; hence, by~\eqref{hh:eq:tighttrace},
\begin{equation}
	\sum_{i=1}^m\bigl\langle A_i,\pi_{\Ee}[A_i,\varphi]\bigr\rangle_H
	=-\sum_{i=1}^m\bigl\langle A_i,P_\varphi A_i\bigr\rangle_H
	=-\lambda^2\,\tr(P_\varphi)=-\lambda^2\,\kae(\varphi) .
\end{equation}

For the first summand, since $\pi_{\Ee}$ is a fixed linear map, the Leibniz rule and~\eqref{eq6a7ee065} give
\begin{equation}
	\sum_{i=1}^m\bigl\langle A_i,\pi_{\Ee}(\tilde X_i\varphi)\bigr\rangle_H
	=\sum_{i=1}^m\tilde X_i\langle A_i,\pi_{\Ee}\varphi\rangle_H
	-\bigl\langle \tr_G(\Diff^2F),\pi_{\Ee}\varphi\bigr\rangle_H .
\end{equation}
 By Lemma~\ref{hh:lem:modular}\ref{hh:it:mod2}, by~\eqref{eq6a7ebdf0} and by~\eqref{eq670d3315}, writing $f_i:=\langle A_i,\pi_{\Ee}\varphi\rangle_H$,
\begin{equation}
\begin{aligned}
	 \sum_{i=1}^m\tilde X_if_i
	&=\div_{\vol_G}\Bigl(\sum_{i=1}^mf_i\tilde X_i\Bigr)+\sum_{i=1}^m\kappa_G(X_i)f_i \\
	&=\div_{\vol_G}\Bigl(\sum_{i=1}^mf_i\tilde X_i\Bigr)+\bigl\langle \Diff F[\mG],\pi_{\Ee}\varphi\bigr\rangle_H ,
\end{aligned}
\end{equation}
where in the last equality we used the linearity of $\Diff F$. 

Adding the two summands, and using $\langle\mH,\pi_{\Ee}\varphi\rangle_H=\kappa_H(\pi_{\Ee}\varphi)$, which holds by~\eqref{hh:eq:mG} on $H$ because $\pi_{\Ee}\varphi\in\Ee$, we get that the left-hand side of~\eqref{hh:eq:key} equals
\begin{align*}
	&\div_{\vol_G}\Bigl(\sum_{i=1}^mf_i\tilde X_i\Bigr)
	-\bigl\langle \tr_G(\Diff^2F)-\Diff F[\mG],\pi_{\Ee}\varphi\bigr\rangle_H
	-\lambda^2\kae(\varphi)\\
	&\qquad=\div_{\vol_G}\Bigl(\sum_{i=1}^mf_i\tilde X_i\Bigr)
	-\bigl\langle \Tt(F),\pi_{\Ee}\varphi\bigr\rangle_H
	+\lambda^2\kappa_H(\pi_{\Ee}\varphi)-\lambda^2\kae(\varphi) ,
\end{align*}
which is~\eqref{hh:eq:key} by~\eqref{hh:eq:chi}.
\end{proof}

\begin{theorem}\label{hh:thm:firstvariation}
Let $G$ and $H$ be sub-Riemannian Lie groups, $\Omega\subset G$ open, and $F:\Omega\to H$ a $C^2$-smooth conformal submersion of factor $\lambda$.
Let $\Qq$ be a complement of $\Ee$ in $\frk h$ and let $\chi$ be the modular mismatch of the splitting $\frk h=\Ee\oplus\Qq$.
Then, for every variation $(F_s)_{s\in(-\epsilon,\epsilon)}$ of $F$ with variation field $\varphi$,
\begin{equation}\label{hh:eq:firstvariation}
	\left.\frac{\diff}{\diff s}\right|_{s=0}\Energy(F_s)
	=\int_\Omega\Bigl(\lambda^2\,\chi(\varphi)-\bigl\langle \Tt(F),\pi_{\Ee}\varphi\bigr\rangle_H\Bigr)\did\vol_G .
\end{equation}
In particular, the right-hand side of~\eqref{hh:eq:firstvariation} does not depend on the choice of $\Qq$, although both its summands do.
\end{theorem}

\begin{proof}
By~\eqref{hh:eq:admissible} we have $\tilde X_i\varphi+[A_i,\varphi]\in\Ee$, and $A_i\in\Ee$, so that
$\langle A_i,\tilde X_i\varphi+[A_i,\varphi]\rangle_H
=\langle A_i,\pi_{\Ee}(\tilde X_i\varphi+[A_i,\varphi])\rangle_H$.
Hence~\eqref{hh:eq:firstvariation} follows by integrating~\eqref{hh:eq:key} over $\Omega$ and comparing with~\eqref{hh:eq:firstvar}:
indeed, the vector field $\sum_if_i\tilde X_i$ has compact support, because so does $\varphi$, and thus the integral of its divergence vanishes, by~\eqref{eq6904c1b2}.
Finally, the left-hand side of~\eqref{hh:eq:firstvariation} does not involve $\Qq$.
\end{proof}

Since harmonic morphisms are the conformal submersions with $\Tt(F)=0$, we obtain Theorem~\ref{thm6a7ee2b1}:

\begin{corollary}[Theorem~\ref{thm6a7ee2b1}]\label{hh:cor:firstvariation}
Let $G$ and $H$ be sub-Riemannian Lie groups, $\Omega\subset G$ open, and $F:\Omega\to H$ a harmonic morphism, with conformal factor $\lambda$ as in Theorem~\ref{thm6a7ae38c}.
Let $\Qq$ be a complement of $\Ee$ in $\frk h$ and let $\chi$ be the modular mismatch of the splitting $\frk h=\Ee\oplus\Qq$.
Then, for every variation $(F_s)_{s\in (\epsilon,\epsilon)}$
 of $F$ with variation field $\varphi$,
\begin{equation}\label{hh:eq:firstvariationHM}
	\left.\frac{\diff}{\diff s}\right|_{s=0}\Energy(F_s)
	=\int_\Omega \lambda^2\,\chi(\varphi)\did\vol_G .
\end{equation}
Consequently, $F$ is a harmonic map if and only if $\int_\Omega\lambda^2\chi(\varphi)\did\vol_G=0$ for every variation field $\varphi$ of $F$.
\end{corollary}

\begin{proof}
By Theorem~\ref{thm6a7ae38c}, $F$ is a smooth conformal submersion of factor $\lambda$, and $\Tt(F)=0$.
Apply Theorem~\ref{hh:thm:firstvariation}.
\end{proof}

\subsection{When harmonic morphisms are harmonic maps}\label{hh:subs:positive}

\begin{theorem}\label{hh:thm:balanced}
Let $G$ and $H$ be sub-Riemannian Lie groups, $\Omega\subset G$ open, and $F:\Omega\to H$ a harmonic morphism.
If the polarized Lie algebra $(\frk h,\Ee)$ admits a modularly balanced complement, then $F$ is a harmonic map.
\end{theorem}

\begin{proof}
Immediate from Corollary~\ref{hh:cor:firstvariation} with $\chi=0$.
\end{proof}

\begin{corollary}\label{hh:cor:examples}
Let $G$ and $H$ be sub-Riemannian Lie groups, $\Omega\subset G$ open, and $F:\Omega\to H$ a harmonic morphism.
If $H$ is Riemannian, or if $H$ is a Carnot group, or, more generally, if the polarized Lie algebra $(\frk h,\Ee)$ is as in Proposition~\ref{hh:prop:examples}, then $F$ is a harmonic map.
\end{corollary}

\section{A harmonic morphism that is not a harmonic map}\label{sec6a8ec490}

\subsection{A harmonic morphism that is not a harmonic map}\label{hh:subs:counterexample}

\begin{definition}\label{hh:def:example}
Let $\frk h$ be the $4$-dimensional real Lie algebra with basis $Y_1,Y_2,Y_3,Z$ and brackets
\begin{equation}\label{hh:eq:brackets}
	[Y_1,Y_2]=Z,
	\qquad
	[Y_3,Y_1]=Y_1,
	\qquad
	[Y_3,Y_2]=Y_2,
	\qquad
	[Y_3,Z]=2Z ,
\end{equation}
all other brackets of basis vectors being zero.
Let $H$ be the corresponding connected and simply connected Lie group, polarized by
\begin{equation}
	\Ee:=\langle Y_1,Y_2,Y_3\rangle ,
\end{equation}
with $Y_1,Y_2,Y_3$ declared orthonormal, and endowed with a left Haar measure $\vol_H$.
\end{definition}

\begin{lemma}\label{hh:lem:example}
The algebra $\frk h$ of Definition~\ref{hh:def:example} is solvable and isomorphic to $\R Y_3\ltimes\frk h_1$, where $\frk h_1=\langle Y_1,Y_2,Z\rangle$ is the $3$-dimensional Heisenberg algebra and $\ad_{Y_3}|_{\frk h_1}$ is the derivation $\operatorname{diag}(1,1,2)$;
equivalently, $\frk h$ is the Iwasawa algebra $\frk a\oplus\frk n$ of $\frk{su}(2,1)$, that is, the Lie algebra of the group of the complex hyperbolic plane $\C H^2$.
The subspace $\Ee$ is a bracket-generating hyperplane, $H$ is not unimodular, and, for the complement $\Qq:=\R Z$,
\begin{equation}\label{hh:eq:examplechi}
	\kappa_H=4Y_3^* ,
	\qquad
	\mH=4Y_3 ,
	\qquad
	\kae=\kaq=2Y_3^* ,
	\qquad
	\chi=2Y_3^*\neq0 ,
\end{equation}
where $Y_3^*$ is the dual basis vector of $Y_3$.
\end{lemma}

\begin{proof}
The endomorphism $\delta:=\operatorname{diag}(1,1,2)$ of $\frk h_1$ is a derivation, because
$\delta[Y_1,Y_2]=2Z=[\delta Y_1,Y_2]+[Y_1,\delta Y_2]$;
hence the semidirect product $\R Y_3\ltimes_\delta\frk h_1$ is a Lie algebra with the brackets~\eqref{hh:eq:brackets}.
The subspace $\Ee$ is bracket generating because $[Y_1,Y_2]=Z$.
In the basis $(Y_1,Y_2,Y_3,Z)$ the matrices of $\ad_{Y_1},\ad_{Y_2},\ad_Z$ have zero diagonal, while $\ad_{Y_3}=\operatorname{diag}(1,1,0,2)$; therefore $\kappa_H=4Y_3^*$ and, by~\eqref{hh:eq:mG} on $H$, $\mH=4Y_3$.
With $\Qq=\R Z$ we have $\pi_{\Ee}\ad_{Y_3}|_{\Ee}=\operatorname{diag}(1,1,0)$, and $\pi_{\Ee}\ad_B|_{\Ee}$ has zero diagonal for $B\in\{Y_1,Y_2\}$, while $\ad_ZY_1=\ad_ZY_2=0$ and $\pi_{\Ee}\ad_ZY_3=\pi_{\Ee}(-2Z)=0$; hence $\kae=2Y_3^*$.
Similarly $\kaq(B)$ is the coefficient of $Z$ in $[B,Z]$, so $\kaq=2Y_3^*$.
Finally $\chi=\kaq\circ\pi_{\Ee}-\kae\circ\pi_{\Qq}=2Y_3^*$, because $\kae(Z)=0$.
\end{proof}

\begin{theorem}\label{hh:thm:counterexample}
Let $H$ be the sub-Riemannian Lie group of Definition~\ref{hh:def:example}, let $\Omega\subset H$ be a bounded open set and let $F:=\Id|_\Omega:\Omega\to H$.
Then:
\begin{enumerate}[label=(\roman*)]
\item\label{hh:it:cex1}
	$F$ is a harmonic morphism, it is a conformal submersion of factor $\lambda\equiv1$, and $\Tt(F)=0$;
\item\label{hh:it:cex2}
	$F$ is \emph{not} a harmonic map: there are variations of $F$ along which the horizontal energy strictly decreases.
\end{enumerate}
\end{theorem}

\begin{proof}
\ref{hh:it:cex1}
The identity is trivially a harmonic morphism.
Moreover $\Diff F(p)=\Id_{\frk h}$ for every $p$, so $\Diff F(p)|_{\Ee}$ is a homothetic projection of factor $1$, and $\Diff^2F=0$;
since the source and the target coincide, $\Tt(F)=0-\mH+\mH=0$, in accordance with Corollary~\ref{cor6a7ec445}.

\ref{hh:it:cex2}
Fix $\varphi_3\in C^\infty_c(\Omega;\R)$ with $\varphi_3\geq0$ and $\varphi_3\not\equiv0$, and define
\begin{equation}\label{hh:eq:variation}
	F_s(p):=p\exp\bigl(s\,\varphi_3(p)Y_3\bigr),
	\qquad s\in\R .
\end{equation}
Each $F_s$ is smooth and agrees with $F$ outside the support of $\varphi_3$.
Write $\psi:=s\varphi_3$ and $c(p):=\exp(\psi(p)Y_3)$.
For $v\in\frk h$ we have
\begin{equation}
	F_s(p)^{-1}F_s(p\exp(tv))
	=\bigl(c(p)^{-1}\exp(tv)c(p)\bigr)\cdot\bigl(c(p)^{-1}c(p\exp(tv))\bigr) ,
\end{equation}
a product of two curves through the identity, whose derivatives at $t=0$ therefore add.
Since $c(p)^{-1}c(p\exp(tv))=\exp\bigl((\psi(p\exp(tv))-\psi(p))Y_3\bigr)$, both factors lying in the one-parameter subgroup generated by $Y_3$, we get
\begin{equation}\label{hh:eq:DFs}
	\Diff F_s(p)[v]=\Ad\bigl(\exp(-\psi(p)Y_3)\bigr)v+(\tilde v\psi)(p)\,Y_3 .
\end{equation}
The map $\Ad(\exp(-\psi Y_3))=e^{-\psi\ad_{Y_3}}$ acts as $e^{-\psi}$ on $\langle Y_1,Y_2\rangle$, as the identity on $Y_3$ and as $e^{-2\psi}$ on $Z$; in particular it preserves $\Ee$.
As also $Y_3\in\Ee$, formula~\eqref{hh:eq:DFs} shows that \emph{$F_s$ is a contact map for every $s$}, so that $(F_s)_s$ is a variation of $F$ in the sense of Section~\ref{subs6a7ee0c8}.
Explicitly,
\begin{equation}
	\Diff F_s[Y_j]=e^{-\psi}Y_j+(\tilde Y_j\psi)Y_3 \quad (j=1,2),
	\qquad
	\Diff F_s[Y_3]=(1+\tilde Y_3\psi)Y_3 ,
\end{equation}
so that the energy density in~\eqref{hh:eq:energy} is
\begin{equation}
	\sum_{i=1}^3\bigl|\Diff F_s[Y_i]\bigr|^2_H
	=2e^{-2\psi}+(1+\tilde Y_3\psi)^2+(\tilde Y_1\psi)^2+(\tilde Y_2\psi)^2 .
\end{equation}
We differentiate at $s=0$, where $\psi=0$ and $\frac{\de\psi}{\de s}=\varphi_3$; the two squares of first derivatives do not contribute, and we obtain
\begin{equation}
	\left.\frac{\diff}{\diff s}\right|_{s=0}\Energy(F_s)
	=\frac12\int_\Omega\bigl(-4\varphi_3+2\,\tilde Y_3\varphi_3\bigr)\did\vol_H
	=\int_\Omega\bigl(\tilde Y_3\varphi_3-2\varphi_3\bigr)\did\vol_H .
\end{equation}
By~\eqref{eq6a7f1a2b} and Lemma~\ref{hh:lem:example},
$\int_\Omega\tilde Y_3\varphi_3\did\vol_H=\kappa_H(Y_3)\int_\Omega\varphi_3\did\vol_H$, and $\kappa_H(Y_3)=4$.
Therefore
\begin{equation}\label{hh:eq:cexvariation}
	\left.\frac{\diff}{\diff s}\right|_{s=0}\Energy(F_s)
	=2\int_\Omega\varphi_3\did\vol_H>0 ,
\end{equation}
in accordance with Corollary~\ref{hh:cor:firstvariation} and $\chi=2Y_3^*$.
Hence $\Energy(F_s)<\Energy(F)$ for $s<0$ small, and $F$ is not a harmonic map.
\end{proof}

\begin{remark}\label{hh:rem:cex}
Some comments on Theorem~\ref{hh:thm:counterexample}.
\begin{enumerate}[label=(\alph*)]
\item
	The variation~\eqref{hh:eq:variation} is explicit and elementary, and the proof is independent of the theory of~\cite{zbMATH07887823}:
	in particular it does not use the compactness and simple connectedness assumptions, nor the dichotomy between regular and singular points.
	It is a statement about the horizontal energy alone.
	This is relevant, because Theorem~\ref{hh:thm:counterexample} rules out the abnormal equation of~\cite[Theorem 4.2(a)]{zbMATH07887823} as well:
	a map that is not a critical point of the energy is not a harmonic map, whichever equation one writes.
\item
	By Corollary~\ref{hh:cor:examples}, no such example exists with a Riemannian, or Carnot, or unimodular target with $\kae=0$.
	It is the interplay between the non-unimodularity of $H$ and the polarization $V(H)$ that produces the phenomenon, and not the non-unimodularity alone.
\item
	The map $F=\Id$ is an isometry preserving the measure;
	the example shows in particular that isometries of sub-Riemannian Lie groups need not be critical points of the horizontal energy.
\end{enumerate}
\end{remark}

\section{The normal equation for harmonic morphisms}\label{sec6a8ec4e4}

\subsection{Comparison with the Grong--Markina equations}\label{hh:subs:GM}

The critical points of the horizontal energy among contact maps are studied in~\cite[Theorem 4.2]{zbMATH07887823} by means of Lagrange multipliers for the constraint~\eqref{hh:eq:admissible}.
In this section we recast the resulting \emph{normal equation} in our notation and we compare it with the key identity~\eqref{hh:eq:key}.

Let $G$ and $H$ be sub-Riemannian Lie groups, $\Omega\subset G$ open, $X_1,\dots,X_m$ an orthonormal basis of $V(G)$, and $F:\Omega\to H$ a smooth contact map with $A_1,\dots,A_m$ as in~\eqref{hh:eq:Ai}.
For $\varphi\in C^\infty(\Omega;\frk h)$ we set
\begin{equation}\label{hh:eq:Lalpha}
	L_F\varphi:\Omega\to\frk h\otimes V(G)^* ,
	\qquad
	(L_F\varphi)[v]:=\tilde v\varphi+[\Diff F[v],\varphi] ,
	\quad\forall v\in V(G),
\end{equation}
which is the derivative~\eqref{hh:eq:dalpha} of $\Diff F$ along a variation with field $\varphi$,
and, for $\eta:\Omega\to\frk h^*\otimes V(G)^*$, writing $\eta_i:=\eta[X_i]$,
\begin{equation}\label{hh:eq:Lstarexplicit}
	(L_F^*\eta)(B)
	:=-\sum_{i=1}^m\tilde X_i\bigl(\eta_i(B)\bigr)
	+\sum_{i=1}^m\kappa_G(X_i)\,\eta_i(B)
	+\sum_{i=1}^m\eta_i\bigl([A_i,B]\bigr) ,
	\qquad \forall B\in\frk h .
\end{equation}


\begin{lemma}\label{hh:lem:adjoint}
The operator $L^*_F$ is the formal adjoint of $L_F$:
for every $\eta:\Omega\to\frk h^*\otimes V(G)^*$ of class $C^1$ and every $\varphi\in C^\infty_c(\Omega;\frk h)$,
\begin{equation}\label{hh:eq:adjointness}
	\int_\Omega\sum_{i=1}^m\eta_i\bigl((L_F\varphi)[X_i]\bigr)\did\vol_G
	=\int_\Omega(L^*_F\eta)(\varphi)\did\vol_G .
\end{equation}
\end{lemma}

\begin{proof}
By the Leibniz rule and~\eqref{eq6a7f1a2b},
\begin{equation}
\begin{aligned}
	\int_\Omega\eta_i(\tilde X_i\varphi)\did\vol_G
	&=\int_\Omega\Bigl(\tilde X_i\bigl(\eta_i(\varphi)\bigr)-(\tilde X_i\eta_i)(\varphi)\Bigr)\did\vol_G\\
	&=\int_\Omega\Bigl(\kappa_G(X_i)\,\eta_i-\tilde X_i\eta_i\Bigr)(\varphi)\did\vol_G .
\end{aligned}
\end{equation}
Summing over $i$ and adding $\sum_i\eta_i([A_i,\varphi])$ gives~\eqref{hh:eq:adjointness}.
\end{proof}

\begin{definition}\label{hh:def:normal}
A smooth contact map $F:\Omega\to H$ is a \emph{normal harmonic map} if there exists $\eta:\Omega\to\frk h^*\otimes V(G)^*$ of class $C^1$, called a \emph{multiplier}, such that
\begin{equation}\label{hh:eq:GMnormal}
	\eta_i(B)=\langle A_i,B\rangle_H
	\quad\forall B\in\Ee,\ i=1,\dots,m,
	\qquad\text{and}\qquad
	L_F^*\eta=0 .
\end{equation}
\end{definition}

\begin{lemma}\label{hh:lem:normalimplieshm}
A normal harmonic map is a harmonic map.
\end{lemma}

\begin{proof}
Let $\eta$ be a multiplier as in~\eqref{hh:eq:GMnormal} and let $\varphi$ be the variation field of a variation of $F$.
By~\eqref{hh:eq:admissible} we have $(L_F\varphi)[X_i]\in\Ee$, hence
$\langle A_i,(L_F\varphi)[X_i]\rangle_H=\eta_i\bigl((L_F\varphi)[X_i]\bigr)$.
By~\eqref{hh:eq:firstvar} and~\eqref{hh:eq:adjointness},
\begin{equation}
	\left.\frac{\diff}{\diff s}\right|_{s=0}\Energy(F_s)
	=\int_\Omega\sum_{i=1}^m\eta_i\bigl((L_F\varphi)[X_i]\bigr)\did\vol_G
	=\int_\Omega(L^*_F\eta)(\varphi)\did\vol_G=0 . 
\end{equation}
\end{proof}

Given a complement $\Qq$ of $\Ee$ in $\frk h$, we write
\begin{equation}\label{hh:eq:flat}
	A^\flat:=\langle A,\pi_{\Ee}(\,\cdot\,)\rangle_H\in\frk h^* ,
	\qquad A\in\Ee ,
\end{equation}
so that $A^\flat|_{\Ee}=\langle A,\cdot\rangle_H$ and $A^\flat|_{\Qq}=0$.
If $\Ann(\Ee)$ 
denotes the space of linear forms on $\frk h$ that vanish on $\Ee$,
the first condition in~\eqref{hh:eq:GMnormal} holds if and only if
\begin{equation}\label{hh:eq:multipliers}
	\eta_i=A_i^\flat+\nu_i
	\qquad\text{with}\qquad
	\nu_i:\Omega\to\Ann(\Ee),\quad i=1,\dots,m ,
\end{equation}
and the $\nu_i$ are uniquely determined by $\eta$.

\begin{corollary}\label{hh:cor:key}
Let $G$ and $H$ be sub-Riemannian Lie groups, $\Omega\subset G$ open, $F:\Omega\to H$ a smooth conformal submersion of factor $\lambda$, let $\Qq$ be a complement of $\Ee$ in $\frk h$ with modular mismatch $\chi$, and let $\eta$ be as in~\eqref{hh:eq:multipliers}.
Then, for every $B\in\frk h$,
\begin{equation}\label{hh:eq:keyGM}
	(L^*_F\eta)(B)
	=-\bigl\langle \Tt(F),\pi_{\Ee}B\bigr\rangle_H
	+\lambda^2\,\chi(B)
	+(L^*_F\nu)(B) .
\end{equation}
\end{corollary}

\begin{proof}
By the linearity of $L^*_F$, it is enough to prove~\eqref{hh:eq:keyGM} for $\nu=0$, that is, for $\eta_i=A_i^\flat$.
Apply the key identity~\eqref{hh:eq:key} to the constant function $\varphi\equiv B$, and observe that
$\langle A_i,\pi_{\Ee}\varphi\rangle_H=A_i^\flat(B)$ and $\tilde X_i\varphi=0$, so that the left-hand side of~\eqref{hh:eq:key} is $\sum_iA_i^\flat([A_i,B])$, while
\begin{equation}
\begin{aligned}
	\div_{\vol_G}\Bigl(\sum_{i=1}^mA_i^\flat(B)\tilde X_i\Bigr)
	&=\sum_{i=1}^m\Bigl(\tilde X_i\bigl(A_i^\flat(B)\bigr)-\kappa_G(X_i)A_i^\flat(B)\Bigr)\\
	&=-(L^*_F\eta)(B)+\sum_{i=1}^mA_i^\flat\bigl([A_i,B]\bigr) ,
\end{aligned}
\end{equation}
by Lemma~\ref{hh:lem:modular}\ref{hh:it:mod2} and~\eqref{hh:eq:Lstarexplicit}.
\end{proof}

\begin{theorem}\label{hh:thm:balancednormal}
Let $G$ and $H$ be sub-Riemannian Lie groups, $\Omega\subset G$ open, and $F:\Omega\to H$ a smooth conformal submersion.
Assume that the polarized Lie algebra $(\frk h,\Ee)$ admits a modularly balanced complement $\Qq$.
Then the following statements are equivalent:
\begin{enumerate}[label=(\roman*)]
\item\label{hh:it:bal1} $F$ is a harmonic morphism;
\item\label{hh:it:bal2} $\Tt(F)=0$, that is, $F$ satisfies~\eqref{eq6a7c3003};
\item\label{hh:it:bal3} $F$ is a normal harmonic map, with the multiplier $\eta$ given by $\eta_i=A_i^\flat$.
\end{enumerate}
\end{theorem}

\begin{proof}
\ref{hh:it:bal1}$\iff$\ref{hh:it:bal2} is Corollary~\ref{cor6a7ec445}.
For $\eta_i=A_i^\flat$, Corollary~\ref{hh:cor:key} with $\nu=0$ and $\chi=0$ gives
$(L^*_F\eta)(B)=-\langle\Tt(F),\pi_{\Ee}B\rangle_H$ for every $B\in\frk h$.
If $\Tt(F)=0$, this vanishes and $\eta$ is a multiplier, because $A_i^\flat|_{\Ee}=\langle A_i,\cdot\rangle_H$ by~\eqref{hh:eq:flat};
conversely, if it vanishes for every $B\in\frk h$, then $\Tt(F)=0$, because $\pi_{\Ee}$ is onto $\Ee$ and $\Tt(F)\in\Ee$.
\end{proof}

The next result shows that the obstruction to the existence of a multiplier is pointwise and purely Lie-algebraic.
Given a complement $\Qq$ of $\Ee$ in $\frk h$, we denote by
$\scr B(A,B):=\pi_{\Qq}[A,B]$, for $A,B\in\Ee$, the Levi form of $(\frk h,\Ee)$.

\begin{lemma}\label{hh:lem:psi}
Let $G$ and $H$ be sub-Riemannian Lie groups, $\Omega\subset G$ open, $F:\Omega\to H$ a smooth conformal submersion of factor $\lambda$, and $\Qq$ a complement of $\Ee$ in $\frk h$.
Let $p\in\Omega$ with $\lambda(p)\neq0$, and let $\nu_1,\dots,\nu_m\in\Ann(\Ee)$.
The following statements are equivalent:
\begin{enumerate}[label=(\roman*)]
\item\label{hh:it:psi1}
	$\displaystyle\sum_{i=1}^m\nu_i\bigl(\scr B(A_i(p),B)\bigr)=-\lambda(p)^2\,\kaq(B)$ for every $B\in\Ee$;
\item\label{hh:it:psi2}
	the linear map $\Psi:\Qq\to\Ee$, $\Psi(Z):=\lambda(p)^{-2}\sum_i\nu_i(Z)A_i(p)$, satisfies
	\begin{equation}\label{hh:eq:ALG}
		\tr\Bigl(\Ee\ni A\longmapsto\Psi\bigl(\scr B(A,B)\bigr)\Bigr)=-\kaq(B)
		\qquad\forall B\in\Ee .
	\end{equation}
\end{enumerate}
Moreover, for $\Psi\in\Hom(\Qq,\Ee)$ the subspace $\Qq_\Psi:=\{Z-\Psi Z:Z\in\Qq\}$ is a complement of $\Ee$ in $\frk h$ with
\begin{equation}\label{hh:eq:kappaPsi}
	\kappa_{\Qq_\Psi}(B)=\kaq(B)+\tr\Bigl(A\mapsto\Psi\bigl(\scr B(A,B)\bigr)\Bigr),
	\qquad B\in\Ee ,
\end{equation}
so that~\eqref{hh:eq:ALG} holds if and only if $\kappa_{\Qq_\Psi}|_{\Ee}=0$.
\end{lemma}

\begin{proof}
\ref{hh:it:psi2}$\THEN$\ref{hh:it:psi1}:
the map $A\mapsto\Psi(\scr B(A,B))=\lambda^{-2}\sum_i\nu_i(\scr B(A,B))A_i$ is a sum of rank-one maps $c_i\otimes A_i$ with $c_i:=\nu_i(\scr B(\,\cdot\,,B))\in\Ee^*$, whose trace is $c_i(A_i)$; hence its trace is $\lambda^{-2}\sum_i\nu_i(\scr B(A_i,B))$.
The converse implication is the same computation.

For~\eqref{hh:eq:kappaPsi}, identify $\Qq_\Psi$ with $\Qq$ by $Z\mapsto Z-\Psi Z$ and observe that, modulo $\Ee$, the projections onto $\Qq$ and onto $\Qq_\Psi$ agree; therefore, for $B\in\Ee$,
\begin{equation}
	\kappa_{\Qq_\Psi}(B)
	=\tr\bigl(Z\mapsto\pi_{\Qq}[B,Z-\Psi Z]\bigr)
	=\kaq(B)-\tr\bigl(\pi_{\Qq}\circ\ad_B\circ\Psi\bigr) .
\end{equation}
On the other hand, writing $S:\Ee\to\Qq$, $S(A):=\scr B(A,B)$, and using $\tr_{\Ee}(\Psi S)=\tr_{\Qq}(S\Psi)$,
\begin{align}
	-\tr\bigl(\pi_{\Qq}\circ\ad_B\circ\Psi\bigr)
	&= \tr\bigl(Z\mapsto\pi_{\Qq}[\Psi Z,B]\bigr) \\
	&= \tr_{\Qq}(S\Psi) 
	= \tr_{\Ee}(\Psi S) 
	= \tr\Bigl(A\mapsto\Psi\bigl(\scr B(A,B)\bigr)\Bigr) ,
\end{align}
and~\eqref{hh:eq:kappaPsi} follows.
\end{proof}

\begin{theorem}\label{hh:thm:obstruction}
Let $G$ and $H$ be sub-Riemannian Lie groups, $\Omega\subset G$ open, and $F:\Omega\to H$ a harmonic morphism with conformal factor $\lambda$ as in Theorem~\ref{thm6a7ae38c}, with $\lambda(p)\neq0$ for some $p\in\Omega$.
If $F$ is a normal harmonic map, then $\Ee$ admits a complement $\Qq'$ in $\frk h$ such that
\begin{equation}\label{hh:eq:necessary}
	\kappa_{\Qq'}|_{\Ee}=0,
	\qquad\text{that is,}\qquad
	\tr\bigl(\pi_{\Qq'}\ad_B|_{\Qq'}\bigr)=0\quad\forall B\in\Ee .
\end{equation}
Condition~\eqref{hh:eq:necessary} depends only on the polarized Lie algebra $(\frk h,\Ee)$: it involves neither $F$, nor $G$, nor the scalar products.
\end{theorem}

\begin{proof}
Fix a complement $\Qq$ of $\Ee$ in $\frk h$, let $\eta$ be a multiplier, written as in~\eqref{hh:eq:multipliers}, and evaluate the identity~\eqref{hh:eq:keyGM} at $p$ on a vector $B\in\Ee$.
The functions $\nu_i(B)$ vanish identically, because $\nu_i\in\Ann(\Ee)$; hence the first two summands of $(L^*_F\nu)(B)$ in~\eqref{hh:eq:Lstarexplicit} vanish, and the third one is $\sum_i\nu_i([A_i,B])=\sum_i\nu_i(\scr B(A_i,B))$, again because $\nu_i$ annihilates $\Ee$.
Since $\Tt(F)=0$ and $\chi|_{\Ee}=\kaq|_{\Ee}$, the equation $L^*_F\eta=0$ gives
\begin{equation}
	0=\lambda(p)^2\kaq(B)+\sum_{i=1}^m\nu_i\bigl(\scr B(A_i(p),B)\bigr)
	\qquad\forall B\in\Ee ,
\end{equation}
which is statement~\ref{hh:it:psi1} of Lemma~\ref{hh:lem:psi}.
The lemma then produces $\Psi\in\Hom(\Qq,\Ee)$ with $\kappa_{\Qq_\Psi}|_{\Ee}=0$.
\end{proof}

\begin{corollary}\label{hh:cor:noPsi}
For the polarized Lie algebra $(\frk h,\Ee)$ of Definition~\ref{hh:def:example} there is \emph{no} complement $\Qq'$ of $\Ee$ in $\frk h$ with $\kappa_{\Qq'}|_{\Ee}=0$.
Consequently, the map $F$ of Theorem~\ref{hh:thm:counterexample} is not a normal harmonic map, and no multiplier as in~\eqref{hh:eq:GMnormal} exists for it.
\end{corollary}

\begin{proof}
By Lemma~\ref{hh:lem:psi} it suffices to show that~\eqref{hh:eq:ALG} has no solution $\Psi\in\Hom(\Qq,\Ee)$, where $\Qq=\R Z$.
Set $P:=\Psi(Z)\in\Ee$ and write $\scr B(A,B)=\beta(A,B)Z$, where $\beta$ is the skew-symmetric bilinear form on $\Ee$ with $\beta(Y_1,Y_2)=1$ and $\beta(Y_3,\cdot)=0$.
The map $A\mapsto\Psi(\scr B(A,B))=\beta(A,B)P$ has trace $\beta(P,B)$, so~\eqref{hh:eq:ALG} reads
\begin{equation}
	\beta(P,B)=-\kaq(B)=-2Y_3^*(B)
	\qquad\forall B\in\Ee .
\end{equation}
Taking $B=Y_3$ gives $\beta(P,Y_3)=-2$, which is impossible because $Y_3$ belongs to the radical of $\beta$, so that $\beta(P,Y_3)=0$.
The last statement follows from Theorem~\ref{hh:thm:obstruction}, or directly from Lemma~\ref{hh:lem:normalimplieshm} and Theorem~\ref{hh:thm:counterexample}\ref{hh:it:cex2}.
\end{proof}

%
\subsection{Two questions}\label{hh:subs:questions}

The first question concerns the gap between the sufficient condition of Theorem~\ref{hh:thm:balancednormal}, that is $\chi=0$, and the necessary condition~\eqref{hh:eq:necessary} of Theorem~\ref{hh:thm:obstruction}, which is the first half of~\eqref{hh:eq:balanced} only.
Notice that in this intermediate case the residual unknown $\nu$ of~\eqref{hh:eq:multipliers} has to solve the underdetermined first-order linear system $(L^*_F\nu)|_{\Qq}=-\lambda^2\chi|_{\Qq}$.

\begin{question}\label{hh:q:gap}
Are there polarized Lie algebras $(\frk h,\Ee)$ that admit a complement $\Qq$ with $\kaq|_{\Ee}=0$ but no modularly balanced complement?
If so, are harmonic morphisms into the corresponding sub-Riemannian Lie groups normal harmonic maps?
\end{question}

The second question is suggested by Remark~\ref{hh:rem:why}: the horizontal energy~\eqref{hh:eq:energy} ignores the measure of the target, while harmonic morphisms do not.

\begin{question}\label{hh:q:functional}
Is there a natural modification of the horizontal energy, involving the Haar measure of the target, whose critical points among contact maps are exactly the maps satisfying $\Tt(F)=0$?
By~\eqref{hh:eq:firstvariationHM}, the required correction of the first variation is $-\int_\Omega\lambda^2\chi(\varphi)\did\vol_G$.
\end{question}

\appendix
\section{Sub-Riemannian submersions}\label{sec6a8e71d5}
\newcommand{\Span}{\operatorname{Span}}

The following lemma is a simple technical result in differential geometry.
\begin{lemma}\label{lem6a7d950c}
	Let $M,N$ be smooth manifolds, and let $X_1,\dots,X_m\in\Gamma(TM)$ and $Y_1,\dots,Y_n\in\Gamma(TN)$ be two families of vector fields 
	(not necessarily linearly independent).
	Let $\tilde V\subset TM$ and $\tilde W\subset TN$ be the sets
	\begin{equation}
		\tilde V = \bigsqcup_{x\in M} \Span_\R\{X_1(x),\dots,X_m(x)\}
		\text{ and }
		\tilde W = \bigsqcup_{y\in M} \Span_\R\{Y_1(y),\dots,Y_n(y)\} .
	\end{equation}
	Let $F:M\to N$ be a $C^\infty$-smooth map with
	\begin{equation}\label{eq6a7d9506}
		\diff F(x)[\tilde V|_x] = \tilde W|_{F(x)}
		\qquad\forall x\in M.
	\end{equation}
	If $\tilde W$ is a bracket generating subbundle of $TM$, 
	then $F$ is a submersion, that is
	\begin{equation}\label{eq6a8dba73}
		\diff F(p)[T_pM] = T_{F(p)}N
		\qquad\forall p\in M.
	\end{equation}
\end{lemma}
\begin{proof}
	Fix $p\in M$. 
	Up to taking an open neighborhood of $F(p)$ in $N$ and selecting a subset of $\{Y_1(y),\dots,Y_n(y)\}$, we assume that $Y_1(y),\dots,Y_n(y)$ are linearly independent for every $y\in N$.
	By~\eqref{eq6a7d9506}, up to reshuffling, we assume
	\begin{equation}\label{eq6a8e6752}
		\tilde W|_{F(p)} = \Span_{\R} \{\diff F(p)[X_1(p)],\dots,\diff F(p)[X_n(p)]\} .
	\end{equation}
	This implies that $\{X_1,\dots,X_n\}$ are linearly independent in a neighborhood of $p$, and, up to shrinking $M$, we assume this neighborhood is $M$.
	By~\eqref{eq6a7d9506}, there are $A_i^k\in C^\infty(M)$ so that
	\begin{equation}\label{eq6a8def7c}
		Y_i(F(x)) = \sum_{k=1}^n A_i^k(x) \diff F(x)[X_k(x)] 
		= \diff F(x)\left[ \sum_{k=1}^n A_i^k(x) X_k(x) \right] 
		\qquad\forall i\in\{1,\dots,n\}.
	\end{equation}
	Let $A_i := \sum_{k=1}^m A_i^k X_k \in\Gamma(TM)$ for all $i\in\{1,\dots,n\}$.
	The identity~\eqref{eq6a8def7c} means that $A_i$ and $Y_i$ are \emph{$F$-related}, as in~\cite[p.~119]{MR1666820} or~\cite[p.~182]{MR2954043}.
	
	
	Let $\frk f_n$ be the free Lie algebra with $n$ generators $f_1,\dots,f_n$.
	Then we have two morphisms of Lie algebras $\phi_M:\frk f_n\to\Gamma(TM)$ and $\phi_N:\frk f_n\to\Gamma(TN)$ so that $\phi_M(f_j)=A_j$ and $\phi_N(f_j) = Y_j$, for every $j\in\{1,\dots,n\}$.
	By~\cite[p~.119]{MR1666820} or~\cite[Prop.~8.30]{MR2954043}, for every $v\in\frk f_n$, $\phi_Mv$ and $\phi_Nv$ are $F$-related,
	i.e.,
	\begin{equation}\label{eq6a8e6bde}
		(\phi_Nv)(F(x)) = \diff F(x)[(\phi_Mv)(x)] 
		\qquad\forall x\in M,\ \forall v\in\frk f_n.
	\end{equation}
	
	Since $\tilde W$ is bracket-generating, $\phi_N(\frk f_n)|_{F(p)} = T_{F(p)}N$. 
	From \eqref{eq6a8e6bde}, we conclude $\diff F(p)[T_pM] = T_{F(p)}N$.

\end{proof}

\begin{remark}\label{rem6a8e65a2}
	Lemma~\ref{lem6a7d950c} requires $\tilde W$ to be a subbundle, that is, $y\mapsto\dim(\tilde W|_y)$ is constant.
	This hypothesis is necessary, as the following example shows.
	Consider $M=N=\R^2$, with the vector fields
	\begin{equation}
	\begin{aligned}
		X_1 = \de_x , &&&X_2 = \de_y ,\\
		Y_1 = \de_x, &&&Y_2 = x\de_y .
	\end{aligned}
	\end{equation}
	Notice that $\{Y_1,Y_2\}$ is bracket generating.
	The map $F:M\to N$, $F(x,y) := (x,x^2y)$, does satisfy~\eqref{eq6a7d9506}, but not~\eqref{eq6a8dba73} on $\{x=0\}$.
\end{remark}

\begin{remark}
	The $C^\infty$ smoothness of $F$ is required in~\eqref{eq6a8e6bde}.
	If $\tilde W$ has step $s$, we can use~\eqref{eq6a8e6bde} only for $v\in\frk f_n$ up to step $s$,
	and thus we only need $F$ to be $C^s$-smooth.
\end{remark}

\printbibliography
\end{document}